\documentclass[reqno]{amsart}

\usepackage{amsmath,amsthm,amstext,amssymb,bbm}
\usepackage{hyperref}
\usepackage{accents}
\usepackage[utf8]{inputenc}
\usepackage{enumerate,enumitem}
\usepackage{amscd,  amsfonts,  amsbsy,  mathrsfs, upgreek, mathtools, stmaryrd, bbm}
\usepackage{pict2e}

\theoremstyle{plain}
\newtheorem{theorem}{Theorem}[section]
\newtheorem{lemma}[theorem]{Lemma}

\newtheorem{corollary}[theorem]{Corollary}
\newtheorem{proposition}[theorem]{Proposition}

\newtheorem*{claim*}{Claim}
\newtheorem*{subclaim*}{Subclaim}
\theoremstyle{definition}
\newtheorem{definition}[theorem]{Definition}
\newtheorem{remarks}[theorem]{Remarks}

\newtheorem{question}[theorem]{Question}

\newcommand{\betrag}[1]{\vert{#1}\vert}

\newcommand{\dom}[1]{{{\rm{dom}}(#1)}}

\newcommand{\ran}[1]{{{\rm{ran}}(#1)}}

\newcommand{\map}[3]{{#1}:{#2}\longrightarrow{#3}}

\newcommand{\Set}[2]{\{{#1}~\vert~{#2}\}}
\newcommand{\seq}[2]{\langle{#1}~\vert~{#2}\rangle}

\newcommand{\goedel}[2]{{\prec}{#1},{#2}{\succ}}

\newcommand{\anf}[1]{{\text{``}\hspace{0.3ex}{#1}\hspace{0.3ex}\text{''}}}

\newcommand{\HH}[1]{{\rm{H}}(#1)}

\newcommand{\Add}[2]{{\rm{Add}}({#1},{#2})}

\newcommand{\id}{{\rm{id}}}

\newcommand{\On}{{\rm{On}}}
\newcommand{\LL}{{\rm{L}}}

\newcommand{\ZFC}{{\rm{ZFC}}}

\newcommand{\BPFA}{{\rm{BPFA}}}

\newcommand{\DDD}{{\mathbb{D}}}
\newcommand{\EEE}{{\mathbb{E}}}

\newcommand{\PPP}{{\mathbb{P}}}

\newcommand{\RRR}{{\mathbb{R}}}

\newcommand{\VV}{{\rm{V}}}

\newcommand{\calA}{\mathcal{A}}

\newcommand{\calL}{\mathcal{L}}

 \newcommand{\CH}{{\rm{CH}}}
 
\newcommand{\diag}{{\begin{picture}(7,7) \put(1,1){\line(1,0){5}} \put(6,1){\line(0,1){5}} \put(1,1){\line(1,1){5}} \end{picture}\hspace{1.1pt}}}

\newcommand{\Loo}{\mathcal{L}_{\omega_1,\omega}}
\newcommand{\M}{\ensuremath{\mathcal{M}}}
\newcommand{\N}{\ensuremath{\mathcal{N}}}
\renewcommand{\aa}{\aleph_{\alpha}} 
\newcommand{\aapo}{\aleph_{\alpha+1}}

\title{Non-absoluteness of Hjorth's Cardinal Characterization}

\author[P. L\"ucke]{Philipp L\"ucke}
\address[Philipp L\"ucke]{Institut de Matem\`{a}tica, Universitat de Barcelona, 
Gran via de les Corts Catalanes 585,
08007 Barcelona, Spain.}
\email{philipp.luecke@ub.edu}

\author[I. Souldatos]{Ioannis Souldatos}
\address[Ioannis Souldatos]{Department of Mathematics,
Aristotle University of Thessaloniki, Thessaloniki 54124, Greece}
\email{souldatos@math.auth.gr}

\thanks{The authors would like to thank Paul Larson for bringing the results of {\cite[Section 6]{MR2298476}} and {\cite[Section 4]{MR792822}} to their attention. 
This project has received funding from the European Union’s Horizon 2020 research and innovation programme under the Marie Sk{\l}odowska-Curie grant agreement No 842082 of the first author (Project \emph{SAIFIA: Strong Axioms of Infinity -- Frameworks, Interactions and Applications}).}

\begin{document}

\begin{abstract} 
In \cite{HjorthsKnightsModel},  Hjorth proved  that for every countable ordinal  $\alpha$, there exists a complete $\Loo$-sentence $\phi_\alpha$ that has models of all cardinalities less than or equal to $\aa$, but no models of cardinality  $\aapo$.  
Unfortunately, his solution does not yield a single $\Loo$-sentence $\phi_\alpha$, but a  set of $\Loo$-sentences, one of which is guaranteed to work. It was conjectured in \cite{CharacterizableCardinals} that it is independent of the axioms of $\ZFC$ which of these sentences has the desired property. 

In the present paper, we prove that this conjecture is true. 
More specifically, we isolate a diagonalization principle for functions from $\omega_1$ to $\omega_1$ which is a consequence of the \emph{Bounded Proper Forcing Axiom} ($\BPFA$) and then we use this principle to  prove that Hjorth's solution to characterizing $\aleph_2$ in models of $\BPFA$ is different than in models of $\CH$. 
In addition, we show that large cardinals are not needed to obtain this independence result by 
proving that 
our diagonalization principle can be forced over models of $\CH$.
%
\end{abstract}

\maketitle


\section{Introduction}

The present paper contributes to the study of the following model-theoretic concepts:

\begin{definition} 
 \begin{enumerate}
     \item An $\Loo$-sentence $\psi$ \emph{characterizes} an infinite cardinal $\kappa$, if $\psi$ has models in all infinite cardinalities less or equal to $\kappa$, but no models in cardinality $\kappa^+$. 
     
     \item A countable model characterizes some cardinal $\kappa$, if the same is true for its Scott sentence. 
 \end{enumerate}
\end{definition}

We are interested in the following question: Given $\alpha<\omega_1$,  is there a complete $\Loo$-sentence $\psi_\alpha$ that characterizes $\aleph_\alpha$? Although the problem is quite easy to solve if we allow the $\Loo$-sentence to be incomplete, it poses a genuine challenge as stated. The question  was answered in the affirmative by  Hjorth who proved the following theorem.

\begin{theorem}[{\cite[Theorem 1.5]{HjorthsKnightsModel}}]\label{hjorth}
 For every $\alpha<\omega_1$, there exists a complete $\Loo$-sentence that characterizes $\aleph_\alpha$.
\end{theorem}

Unfortunately, Hjorth's solution is unsatisfactory. As observed in \cite{BKLdap}, for every countable ordinal $\alpha$,  Hjorth produces not a single $\Loo$-sentence, but a whole set $S_\alpha$ of $\Loo$-sentences.\footnote{There is no mention of a set $S_\alpha$ in Hjorth's original proof. This notation was introduced in \cite{BKLdap}.} In \cite{BKLdap}, the authors notice that if $\alpha$ is a  finite ordinal, then the set $S_\alpha$ is finite. Otherwise, they state that the set $S_\alpha$  of  $\Loo$-sentences can be chosen to be countable. Below, we will argue that it is possible to find a finite set $S_\alpha$ of $\calL_{\omega_1,\omega}$-sentences with the desired properties for every   countable ordinal $\alpha$.  Hjorth's proof then shows that at least one of the sentences in the set $S_\alpha$  characterizes the cardinal $\aleph_\alpha$, but it provides no evidence which element of $S_\alpha$ has this property.  

To see why this is the case, we briefly explain Hjorth's construction behind Theorem \ref{hjorth}. First, assume that for some countable ordinal $\alpha$, there is a countable model $\M$ whose Scott sentence characterizes $\aleph_\alpha$.   Working by induction, Hjorth wants to create another countable model whose Scott sentence characterizes 
$\aleph_{\alpha+1}$. 
To achieve this, he first defines a countable model, which he calls \emph{$(\M,\N)$-full}, 
using $\M$ and what we will call the \emph{first Hjorth construction}.\footnote{In Hjorth's proof, the existence of the model 
$\M$ is an assumption. The existence of $\N$ is something that comes out of the proof.}  Hjorth proves that the Scott 
sentence of this 
$(\M,\N)$-full model characterizes either $\aleph_\alpha$ or $\aleph_{\alpha+1}$. If the latter is the case, we are done. 
Otherwise, Hjorth proceeds one more round to use the $(\M,\N)$-full model from the first step and what we will call the \emph{second Hjorth construction}. If the $(\M,\N)$-full model characterizes some $\kappa$, then Hjorth's second construction 
characterizes $\kappa^+$. In particular, if the $(\M,\N)$-full model characterizes $\aleph_\alpha$, then the second Hjorth 
construction produces a model that characterizes $\aleph_{\alpha+1}$. Notice here that the failure of the $(\M,\N)$-full model to characterize 
$\aleph_{\alpha+1}$ is used to prove that the second Hjorth construction does indeed characterize $\aleph_{\alpha+1}$.  %
In either case, there exists some $\Loo$-sentence that characterizes 
$\aleph_{\alpha+1}$ and the induction step is complete.

At limit stages, Hjorth takes disjoint unions of the previously constructed models. For instance for $\alpha=\omega$, Hjorth considers the disjoint union of countable models $\M_n$, $n<\omega$, where each $\M_n$ characterizes $\aleph_n$. This union characterizes $\aleph_\omega$, but, as we mentioned, we do not know which the models $\M_n$ are.

Since at successor stages we have to choose between the first and the second Hjorth construction and we repeat this process for every countable successor ordinal, the result is a binary tree of $\Loo$-sentences of height $\omega_1$. 
The $\alpha^{th}$ level of the tree gives us the set $S_\alpha$. In particular, at least one of the sentences at the $\alpha^{th}$ level of the tree characterizes $\aa$.

We now briefly observe that we can do slightly better for $\alpha=\omega$ and for each countable limit ordinal $\alpha$ in general. Consider the following countable models: $\M_0$ is a countable model which characterizes $\aleph_0$ and $\M_{n+1}$ is the second Hjorth construction which inductively uses $\M_n$ as input. If $\M_n$ characterizes some $\aleph_m$, then we mentioned that the first Hjorth construction characterizes either $\aleph_m$ or $\aleph_{m+1}$. It follows that the second Hjorth construction characterizes either $\aleph_{m+1}$ or $\aleph_{m+2}$. Therefore, we can prove inductively that for each $n<\omega$, the model $\M_{n+1}$ characterizes some $\aleph_k$ for $n< k\le 2(n+1)$. Although we will prove that the value of $k$ is independent of $\ZFC$, the disjoint union of the $\M_n$'s always characterizes $\aleph_\omega$. In other words, we can isolate one $\Loo$-sentence that belongs to the set $S_\omega$ and which provably characterizes $\aleph_\omega$.  
A similar argument applies to all countable limit ordinals $\alpha$: There exists one $\Loo$-sentence that belongs to the set $S_\alpha$ and which provably characterizes $\aleph_\alpha$. This greatly reduces the complexity of the tree whose levels are the $S_\alpha$'s, as we can assume that for limit $\alpha$ the set $S_\alpha$ is a singleton.

The problem for successor $\alpha$ remains and in \cite{CharacterizableCardinals}, it was conjectured that it is independent of the axioms of $\ZFC$ whether the Scott sentence of the 
$(\M,\N)$-full model characterizes $\aleph_\alpha$ or $\aleph_{\alpha+1}$. This would imply that it is independent of  $\ZFC$  whether the first or the second Hjorth construction characterizes $\aleph_{\alpha+1}$. Some evidence towards the validity of this conjecture was given by  the following result.

\begin{theorem}[{\cite[Theorem 2.20]{CharacterizableCardinals}}]\label{nofull}
 If $\M$ is a countable model that characterizes $\aa$ and $\aleph_{\alpha}^\omega=\aleph_{\alpha}$, then there is no $(\M,\N)$-full structure of size 
$\aleph_{\alpha+1}$.  
\end{theorem}

In contrast, for the case $\alpha=0$, Hjorth proves 
that the $(\M,\N)$-full model characterizes $\aleph_1$. The proof works both under $\CH$ and its negation, but it uses 
results from descriptive set  theory that cannot be used to prove the statement for cardinals bigger than $\aleph_1$.

The purpose of this paper is to prove  the above conjecture 
by showing that the axioms of $\ZFC$ do not answer the given question for $\alpha=1$. 
By Theorem \ref{nofull}, it is relatively consistent that for every  countable model $\M$ that characterizes $\aleph_1$, there is no $(\M,\N)$-full structure of size $\aleph_2$. In the following, we will prove that the negation of this statement is also relatively consistent. 
 %
In Section \ref{reform},  we discuss a result from \cite{CharacterizableCardinals} that shows that the existence of an $(\M,\N)$-full structure of size $\aleph_{\alpha+1}$ is equivalent to the existence of a colouring of the two-element subsets of $\omega_{\alpha+1}$ with $\aleph_\alpha$-many colors that possesses certain almost disjointedness and genericity properties.  
In Section \ref{diagonal}, we isolate a combinatorial principle $(\diag)$ and  prove that it is a consequence of the \emph{Bounded Proper Forcing Axiom} $\BPFA$ (see \cite{MR1324501}). In Section \ref{nobpfa}, we then show that the consistency of $\ZFC+(\diag)$ can be established from the consistency of $\ZFC$ alone\footnote{Note that the results of \cite{MR1324501} show that the consistency strength of $\BPFA$ lies strictly between the existence of an inaccessible cardinal and the existence of a Mahlo cardinal.} by showing that the principle $(\diag)$ can be forced over models of $\CH$. Finally, in Section \ref{coloring}, we prove that the  principle $(\diag)$ implies the existence of an $(\M,\N)$-full structure of size $\aleph_{2}$. 
We then end the paper by discussing possibilities to characterize cardinals in an \emph{absolute} way and we propose two ways to formulate this concept in a mathematically sound way.


\section{A reformulation} \label{reform}

In the following, we recall the statement of a result in  \cite{CharacterizableCardinals} which  provides an equivalent condition to the 
existence of an $(\M,\N)$-full model of size $\aleph_{\alpha+1}$. Since  it is easier to work with  this equivalent condition, we will use it for the rest of this paper.

We start with some notation conventions. Given a set $d$, we let $[d]^2$ denote the set of all two-element subsets of $d$ and we let $[d]^{{<}\omega}$ denote the set of all finite subsets of $d$.  Moreover, if $c$ is a function whose domain is of the form $[d]^2$ for some set $d$, then we abbreviate $c(\{x,y\})$ by $c(x,y)$. In addition, given such a function $c$ with domain $[d]^2$ and $x,y\in d$ with $x\neq y$, we define $$\calA^c_{x,y} ~ = ~ \Set{z\in d\setminus\{x,y\}}{c(x,z)=c(y,z)}$$ to be the corresponding \emph{set of agreements}. Finally, given sets $d\subseteq  d'$, a function $c$ with domain $[d]^2$ and a function $c'$ with domain $[d']^2$ and $c'\restriction[d]^2=c$, we say that \emph{$c'$ introduces no new agreement over $c$}, if 
 $\calA^c_{x,y}=\calA^{c'}_{x,y}$ holds for all $x,y\in d$.

\begin{lemma}[{\cite[Theorem 5.1]{CharacterizableCardinals}}]\label{EquivalentLemma} 
 Assume that $\M$ is a countable model that characterizes $\aleph_{\alpha}$.  Then the  following statement are equivalent:
 \begin{enumerate}
  \item There exists an $(\M,\N)$-full structure of cardinality $\aleph_{\alpha+1}$. 

  \item There exists a function  $\map{c}{[\omega_{\alpha+1}]^2}{\omega_{\alpha}}$ and a function $\map{r}{\omega_{\alpha+1}}{\omega_{\alpha+1}}$ with the following properties: 
   \begin{enumerate}
	
	\item\label{item:FiniteAgree} (Finite agreement) For all $\beta<\gamma<\omega_{\alpha+1}$, the set $\calA^c_{\beta,\gamma}$ is finite.\footnote{Note that we allow the possibility that the set $\calA^c_{\beta,\gamma}$ is empty.}
	
	\item\label{item:FiniteClosure} (Finite closure) For every  $a\in[\omega_{\alpha+1}]^{{<}\omega}$, there is  $a\subseteq b\in[\omega_{\alpha+1}]^{{<}\omega}$ that is closed under $\calA^c$, {i.e.} for all $\beta,\gamma\in b$ with $\beta\neq\gamma$, we have  $\calA^c_{\beta,\gamma}\subseteq b$. 
	
	\item\label{item:FiniteExt} (Finite extension) If $d$ is a finite set and $\map{e}{[d]^2}{\omega_{\alpha}}$ is a function with $e\restriction[d\cap\omega_{\alpha+1}]^2=c\restriction[d\cap\omega_{\alpha+1}]^2$ that introduces no new agreements over $c\restriction[d\cap\omega_{\alpha+1}]^2$, then there exists an injection $\map{\iota}{d}{\omega_{\alpha+1}}$ with $\iota\restriction(d\cap\omega_{\alpha+1})=\id_{d\cap\omega_{\alpha+1}}$ and $e(\beta,\gamma)=c(\iota(\beta),\iota(\gamma))$ for all $\beta,\gamma\in d$ with $\beta\neq\gamma$. 

	\item\label{item:Color} (Colouring) The statement \eqref{item:FiniteExt}  holds true even if $d$ is \emph{colored}, {i.e.} if, in addition, there exists a  function $\map{s}{d}{\omega_{\alpha+1}}$ with $s\restriction (d\cap\omega_{\alpha+1})=r\restriction(d\cap\omega_{\alpha+1})$, then there exists an injection $\map{\iota}{d}{\omega_{\alpha+1}}$ with the above properties that also satisfies  $s(\beta)=r(\iota(\beta))$ for all $\beta\in d$.
	\end{enumerate}
\end{enumerate}
\end{lemma}

\begin{remarks}
 \begin{enumerate}
  \item The requirement in \eqref{item:FiniteExt} that ``\emph{$e$ introduces no new agreements over $c\restriction[d\cap\omega_{\alpha+1}]^2$}'' was erroneously omitted in \cite{CharacterizableCardinals}, but is needed for the equivalence. 
  
  \item Since $(\M,\N)$-full structures are defined as {F}ra\"{\i}ss\'{e} limits, they are sufficiently \emph{generic}. The  finite extension property \eqref{item:FiniteExt} is an expression of this genericity. 
  
  \item It follows from \eqref{item:FiniteExt} that the function $c$ is surjective.  Given $\beta< \omega_\alpha$, consider the unique function $\map{e}{[\{\omega_{\alpha+1},\omega_{\alpha+1}+1\}]^2}{\omega_\alpha}$ with the property that  $e(\omega_{\alpha+1},\omega_{\alpha+1}+1)=\beta$. Then the assumptions of \eqref{item:FiniteExt} are satisfied and we find an injection $\map{\iota}{\{\omega_{\alpha+1},\omega_{\alpha+1}+1\}}{\omega_{\alpha+1}}$ with the property that  $c(\iota(\omega_{\alpha+1}),\iota(\omega_{\alpha+1}+1))=\beta$. 
  
  \item The existence of the coloring function $r:\omega_{\alpha+1}\rightarrow\omega_{\alpha+1}$ is important for Hjorth's argument, but once one has constructed a function $\map{c}{[\omega_{\alpha+1}]^2}{\omega_\alpha}$ satisfying the statements \eqref{item:FiniteAgree}, \eqref{item:FiniteClosure} and \eqref{item:FiniteExt} of the above theorem, it is easy to modify this function to also obtain a function $\map{r}{\omega_{\alpha+1}}{\omega_{\alpha+1}}$ such that statement \eqref{item:Color} holds too. 
 \end{enumerate}
\end{remarks}

\section{A diagonalization principle}\label{diagonal}

We now introduce and study a diagonalization principle for families of functions from $\omega_1$ to $\omega_1$ that will be central for the independence results of this paper. 
In Section \ref{coloring}, we will use Lemma  \ref{EquivalentLemma} to prove that this principle implies the existence of an 
$(\M,\N)$-full structure of size $\aleph_2$.

\begin{definition}
 \begin{enumerate}
     \item Given a set $X$, we say that a map $\map{m}{[X]^{{<}\omega}}{[X]^{{<}\omega}}$ is \emph{monotone} if $a\subseteq m(a)$ holds for every finite subset $a$ of $X$. 
     
     \item We let $(\diag)$ denote the statement that for every sequence $\seq{f_\alpha}{\alpha<\omega_1}$ of functions from $\omega_1$ to $\omega_1$, every finite subset $F$ of $\omega_1$ and every monotone function $\map{m}{[\omega_1]^{{<}\omega}}{[\omega_1]^{{<}\omega}}$, there exists a function $\map{g}{\omega_1}{\omega_1}$ such that $F\cap\ran{g}=\emptyset$ and for every $a\in[\omega_1]^{{<}\omega}$, there exists $a\subseteq b\in[\omega_1]^{{<}\omega}$ with the property that  $$\Set{\beta<\omega_1}{f_\alpha(\beta)=g(\beta)} ~ \subseteq ~ m(b)$$ holds for all $\alpha\in m(b)$.
 \end{enumerate}
\end{definition}

 A short argument shows that $(\diag)$  is not provable in $\ZFC$:

\begin{proposition} \label{noch}
 If $(\diag)$ holds, then $2^{\aleph_0}>\aleph_1$. 
\end{proposition}

\begin{proof}
 If $2^{\aleph_0}=\aleph_1$ holds, then there exists a sequence $\seq{\map{f_\alpha}{\omega_1}{\omega}}{\alpha<\omega_1}$ of functions with the property that the set $\Set{f_\alpha\restriction\omega}{\alpha<\omega_1}$ contains all functions from $\omega$ to $\omega$. In particular, for every function $\map{g}{\omega_1}{\omega_1}$, there exists $\alpha<\omega_1$ with $g(n)=f_\alpha(n)$ for all $n<\omega$. 
\end{proof}

In addition, it is possible to use results of Baumgartner in \cite{MR401472} to show that $(\diag)$ is also not a theorem of $\ZFC+\neg\CH$:

\begin{proposition}
 If $\CH$ holds and $G$ is $\Add{\omega}{\omega_2}$-generic over $\VV$, then $(\diag)$ fails in $\VV[G]$. 
\end{proposition}

\begin{proof}
 Work in $\VV[G]$ and assume, towards a contradiction, that $(\diag)$ holds. 
 We can then use this assumption to construct a sequence $\seq{\map{f_\gamma}{\omega_1}{\omega_1}}{\gamma<\omega_2}$ of functions with the property that for all $\delta<\gamma<\omega_2$, the set $\Set{\alpha<\omega_1}{f_\gamma(\alpha)=f_\delta(\alpha)}$ is finite. 
 By considering the graphs of these functions und using a bijection between $\omega_1\times\omega_1$ and $\omega_1$, we can now construct a sequence $\seq{A_\gamma}{\gamma<\omega_2}$ of unbounded subsets of $\omega_1$ with the property that for all $\delta<\gamma<\omega_2$, the set $A_\gamma\cap A_\delta$ is finite. 
 But this contradicts results in {\cite[Section 6]{MR401472}} that show that no  sequence of subsets with these properties exists in $\VV[G]$.  
\end{proof}

 In the remainder of this section, we  use results of Larson to prove that the principle $(\diag)$ is a consequence of forcing axioms. 
 These arguments rely on the following forcing notion defined in {\cite[Section 6]{MR2298476}}:

\begin{definition}
  We let $\DDD$ denote the partial order defined by the following clauses: 
\begin{enumerate}
 \item A condition in $\DDD$ is a triple $p=\langle a_p,\mathscr{F}_p,\mathscr{X}_p\rangle$ such that the following statements hold: 
 \begin{enumerate}
  \item $a_p$ is a function from a finite subset $d_p$ of $\omega_1$ into $\omega_1$. 
  
  \item $\mathscr{F}_p$ is a finite set of functions from $\omega_1$ to $\omega_1$. 
  
  \item $\mathscr{X}_p$ is a finite $\in$-chain of countable elementary submodels of $\HH{\omega_2}$. 
  
  \item If $X\in\mathscr{X}_p$ and $\alpha\in d_p\cap X$, than $a_p(\alpha)\in X$. 
  
  \item If $X\in\mathscr{X}_p$, $\alpha\in d_p\setminus X$ and $f\in X$ is a function from $\omega_1$ to $\omega_1$, then $a_p(\alpha)\neq f(\alpha)$. 
 \end{enumerate}
 
 \item Given conditions $p$ and $q$ in $\DDD$, we have $p\leq_\DDD q$ if and only if the following statements hold: 
  \begin{enumerate}
   \item $d_q\subseteq d_p$, $a_q=a_p\restriction d_q$, $\mathscr{F}_q\subseteq\mathscr{F}_p$ and $\mathscr{X}_q\subseteq\mathscr{X}_p$. 
   
   \item If $\alpha\in d_p\setminus d_q$ and $f\in\mathscr{F}_q$, then $a_p(\alpha)\neq f(\alpha)$. 
  \end{enumerate}
\end{enumerate}
\end{definition}

 Given $\alpha<\omega_1$, we define $D_\alpha$ to be the set of all conditions $p$ in $\DDD$ with $\alpha\in d_p$.

 \begin{proposition}\label{proposition:TotalFunction}
  If $q$ is a condition in $\DDD$ and $\alpha<\omega_1$, then there is a condition $p$ in $D_\alpha$ with  $p\leq_\DDD q$, $\mathscr{F}_p=\mathscr{F}_q$ and $\mathscr{X}_p=\mathscr{X}_q$.
 \end{proposition}
 
 \begin{proof}
 First, assume that $\alpha\in X$ for some $X\in\mathscr{X}_q$. Let $Y\in\mathscr{X}_q$ be $\in$-minimal with this property. Then there exists $\beta\in Y\cap\omega_1$ with  $f(\alpha)\neq\beta$ for all $f\in\mathscr{F}_q$ and $g(\alpha)<\beta$ for all $X\in\mathscr{X}_q\cap Y$ and every function $\map{g}{\omega_1}{\omega_1}$ in $X$. Define $$p ~ = ~ \langle a_q\cup\{\langle \alpha,\beta\rangle\},\mathscr{F}_q,\mathscr{X}_q\rangle.$$ Then $p$ is a condition in $\DDD$ that is an element of $D_\alpha$. Moreover, this construction ensures that $p\leq_\DDD q$ holds. 
 
 Now, assume that $\alpha\notin X$ for all $X\in\mathscr{X}_q$. Pick $\beta<\omega_1$ with $f(\alpha)\neq\beta$ for all $f\in\mathscr{F}_q$ and $g(\alpha)<\beta$  for all $X\in\mathscr{X}_q$ and every function $\map{g}{\omega_1}{\omega_1}$ in $X$. If we define $p$ as above, then we again obtain a condition in $D_\alpha$ below $q$. 
 \end{proof}

 Next, for every function $\map{f}{\omega_1}{\omega_1}$, we let $D_f$ denote the set of all conditions $p$ in $\DDD$ with $f\in\mathscr{F}_p$.

\begin{proposition}\label{proposition:AddFunctions}
 If $q$ is a condition in $\DDD$ and $\map{f}{\omega_1}{\omega_1}$ is a function, then $\langle a_q,\mathscr{F}_q\cup\{f\},\mathscr{X}_q\rangle$ is a condition in $\DDD$ below $q$ that is an element of $D_f$.  \qed 
\end{proposition}

\begin{lemma}[{\cite[Theorem 6.2]{MR2298476}}]
 The partial order $\DDD$ is proper. 
\end{lemma}

We are now ready to show that $(\diag)$ is a consequence of $\BPFA$.  Our proof relies on the following  classical result of Bagaria  that characterizes the validity of $\BPFA$ in terms of generic absoluteness.

\begin{theorem}[{\cite[Theorem 5]{MR1773776}}]
 The following statements are equivalent: 
 \begin{enumerate}
  \item $\BPFA$ holds. 
  
  \item If $\varphi(v)$ is a $\Sigma_1$-formula,\footnote{See {\cite[p. 5]{MR1994835}} for the definition of the \emph{Levy hierarchy} of formulas. Note that, using a \emph{universal} $\Sigma_1$-formula, it is possible to phrase this statement as a single sentence in the language of set theory.} $z$ is an  element of $\HH{\omega_2}$, $\PPP$ is a proper forcing  and $p$ is a condition  in $\PPP$ with   $p\Vdash_\PPP\varphi(\check{z})$, then $\varphi(z)$ holds. 
 \end{enumerate}
\end{theorem}



\begin{theorem}\label{main1}
 $\BPFA$ implies that $(\diag)$ holds.
\end{theorem}

\begin{proof}
 Assume that $\BPFA$ holds and fix a sequence $\vec{f}=\seq{\map{f_\alpha}{\omega_1}{\omega_1}}{\alpha<\omega_1}$ of functions, a finite subset $F$ of $\omega_1$ and a monotone function $\map{m}{[\omega_1]^{{<}\omega}}{[\omega_1]^{{<}\omega}}$. 
 Given $\alpha<\omega_1$, let $\map{c_\alpha}{\omega_1}{\omega_1}$ denote the constant function with value $\alpha$. Moreover, define $p_F=\langle\emptyset,\Set{c_\alpha}{\alpha\in F},\emptyset\rangle$. Then $p_F$ is a condition in $\DDD$. 
 Finally, given $d\in[\omega_1]^{{<}\omega}$, let $E_d$ denote the set of all conditions $p$ in $\DDD$ with the property that there exists $d\subseteq e\in[\omega_1]^{{<}\omega}$ with $d_p=m(e)$ and $f_\alpha\in\mathscr{F}_p$ for all $\alpha\in m(e)$.

 \begin{claim*}
  For every $d\in[\omega_1]^{{<}\omega}$, the set $E_d$ is dense in $\DDD$. 
 \end{claim*}
 
 \begin{proof}[Proof of the Claim]
  Fix a condition $r$ in $\DDD$ and set $e=d\cup d_r\in[\omega_1]^{{<}\omega}$. 
  Since $d_r\subseteq e\subseteq m(e)$, we can now use Proposition \ref{proposition:TotalFunction} to find a condition $q$ in $\DDD$ with $q\leq_\DDD r$ and $d_q=m(e)$. Finally, an application of Proposition \ref{proposition:AddFunctions} yields a condition $p$ in $\DDD$ with $p\leq_\DDD q$, $d_p=d_q=m(e)$ and $f_\alpha\in\mathscr{F}_p$ for all $\alpha\in m(e)$. We then have $p\leq_\DDD r$ and $p\in E_d$. 
 \end{proof}

  Now, let $G$ be $\DDD$-generic over the ground model $\VV$ with $p_F\in G$. Work in $\VV[G]$ and define $g=\bigcup\Set{a_p}{p\in G}$. 
  Then Proposition \ref{proposition:TotalFunction} ensures that $g$ is a function from $\omega_1$ to $\omega_1$. 
  
  \begin{claim*}
   $F\cap\ran{g}=\emptyset$. 
  \end{claim*}
  
  \begin{proof}[Proof of the Claim]
   Fix $\alpha<\omega_1$ and $\beta\in F$. By Proposition \ref{proposition:TotalFunction}, we can find $p\in G$ with $p\leq_\DDD p_F$ and $\alpha\in d_p$. 
   Since $\alpha\notin d_{p_F}=\emptyset$ and $c_\beta\in\mathscr{F}_{p_F}$, the definition of $\DDD$  ensures that $g(\alpha)=a_p(\alpha)\neq c_\beta(\alpha)=\beta$. 
  \end{proof}

  \begin{claim*}
   For every $d\in[\omega_1]^{{<}\omega}$, there exists $d\subseteq e\in[\omega_1]^{{<}\omega}$ with the property that  $$\Set{\beta<\omega_1}{f_\alpha(\beta)=g(\beta)} ~ \subseteq ~ m(e)$$ holds for all $\alpha\in m(e)$. 
  \end{claim*}

   \begin{proof}[Proof of the Claim]
    Using our first claim, we can find $q\in E_d\cap G$. Then there exists $d\subseteq e\in[\omega_1]^{{<}\omega}$ with $d_q=m(e)$. Fix $\alpha\in m(e)$ and $\beta\in\omega_1\setminus m(e)$. 
    By Proposition \ref{proposition:TotalFunction}, there exists $p\in G$ with $p\leq_\DDD q$ and $\beta\in d_p$. 
    Since $\beta\in d_p\setminus d_q$ and $f_\alpha\in\mathscr{F}_q$, the definition of $\DDD$ now implies that $g(\beta)=a_p(\beta)\neq f_\alpha(\beta)$.  
   \end{proof}

   The above claims show that, in $\VV[G]$, there exists a function $\map{g}{\omega_1}{\omega_1}$ with $F\cap\ran{g}=\emptyset$ and the property that for all $d\in[\omega_1]^{{<}\omega}$, there is $d\subseteq e\in[\omega_1]^{{<}\omega}$ such that $\Set{\beta<\omega_1}{f_\alpha(\beta)=g(\beta)}  \subseteq  m(e)$ holds for all $\alpha\in m(e)$.
   Since this statement can be formulated by a $\Sigma_1$-formula with parameters $\vec{f},F,m\in\HH{\omega_2}^\VV$, we can use {\cite[Theorem 5]{MR1773776}} to conclude that  the given  statement also holds in $\VV$. 
\end{proof}


\section{Forcing without large cardinals} \label{nobpfa}

In this section, we prove the following result that shows that 
no large cardinals are needed to establish the  consistency of the principle $(\diag)$.

\begin{theorem}\label{theorem:ConZFC}
 If $\CH$ holds, then there is a proper partial order $\PPP$ that satisfies the $\aleph_2$-chain condition such that  $\mathbbm{1}_\PPP\vdash(\diag)$. 
\end{theorem}

 Following the arguments in  {\cite[Section 4]{MR792822}}, we now  introduce a \emph{matrix version} of Larson's forcing $\DDD$. We then use techniques developed in   {\cite[Section VIII.2]{MR675955}} to show that the constructed partial order possesses the properties listed in Theorem \ref{theorem:ConZFC}.

\begin{definition}
  We let $\EEE$ denote the partial order defined by the following clauses: 
\begin{enumerate}
 \item A condition in $\EEE$ is a triple $p=\langle a_p,\mathscr{F}_p,t_p\rangle$ such that the following statements hold: 
 \begin{enumerate}
  \item $a_p$ is a function from a finite subset $d_p$ of $\omega_1$ into $\omega_1$. 
  
  \item $\mathscr{F}_p$ is a finite set of functions from $\omega_1$ to $\omega_1$. 
  
  \item $t_p$ is a function from a finite $\in$-chain $\mathscr{C}_p$ of countable transitive sets to  the set of non-empty finite subsets of $\HH{\omega_2}$. 
  
  \item If $M\in\mathscr{C}_p$ and $X\in t_p(M)$, then $X$ is a countable  elementary submodel of $\HH{\omega_2}$ and  $M$ is the transitive collapse of $X$. 
  
  \item If $M,N\in\mathscr{C}_p$ with $M\in N$ and $X\in t_p(M)$, then there is $Y\in t_p(N)$ with $X\in Y$. 
  
  \item If $M\in\mathscr{C}_p$,  $X\in t_p(M)$ and $\alpha\in d_p\cap X$, than $a_p(\alpha)\in X$. 
  
  \item If $M\in\mathscr{C}_p$,  $X\in t_p(M)$, $\alpha\in d_p\setminus X$ and $f\in X$ is a function from $\omega_1$ to $\omega_1$, then $a_p(\alpha)\neq f(\alpha)$. 
 \end{enumerate}
 
 \item Given conditions $p$ and $q$ in $\EEE$, we have $q\leq_\EEE p$ if and only if the following statements hold: 
  \begin{enumerate}
   \item $d_p\subseteq d_q$, $a_p=a_q\restriction d_p$, $\mathscr{F}_p\subseteq\mathscr{F}_q$, $\mathscr{C}_p\subseteq\mathscr{C}_q$ and $t_p(M)\subseteq t_q(M)$ for all $M\in\mathscr{C}_p$. 
   
   \item If $\alpha\in d_q\setminus d_p$ and $f\in\mathscr{F}_p$, then $a_q(\alpha)\neq f(\alpha)$. 
  \end{enumerate}
\end{enumerate}
\end{definition}

In order to prove that the partial order $\EEE$ is proper, we start by showing that {\cite[Lemma 6.1]{MR2298476}} can directly be adapted to the matrix forcing $\EEE$, using the same proof.

\begin{lemma}\label{lemma:LarsonUltraSort}
 Let $p$ be a condition in $\EEE$ and let $D$ be a subset of $\EEE$  that is dense below $p$. 
 Then there exists $\lambda<\omega_1$ with the property that for every finite set $\mathscr{F}$ of functions from $\omega_1$ to $\omega_1$, there exists $q\in D$ below $p$ with $a_q\subseteq\lambda\times\lambda$ and $a_q(\alpha)\neq f(\alpha)$ for all $\alpha\in d_q\setminus d_p$ and $f\in\mathscr{F}$. 
\end{lemma}

\begin{proof}
 Assume, towards a contradiction, that the above statement fails. 
 Then there exists a sequence $\seq{\mathscr{F}_\lambda}{\lambda<\omega_1}$ of finite, non-empty sets of functions from $\omega_1$ to $\omega_1$ with the property that for every $\lambda<\omega_1$ and every $q\in D$ below $p$ with $a_q\subseteq\lambda\times\lambda$, there exists $\alpha^\lambda_q\in d_q\setminus d_p$ and $f^\lambda_q\in\mathscr{F}_\lambda$ with $a_q(\alpha^\lambda_q)=f^\lambda_q(\alpha^\lambda_q)$. 
 Without loss of generality, we may assume that there is $0<n<\omega$ with $\betrag{\mathscr{F_\lambda}}=n$ for all $\lambda<\omega_1$. 
 Given $\lambda<\omega_1$, pick functions $f^\lambda_0,\ldots,f^\lambda_{n-1}$ from $\omega_1$ to $\omega_1$ with $\mathscr{F}_\lambda=\{f^\lambda_0,\ldots,f^\lambda_{n-1}\}$. 

 Now, fix a uniform ultrafilter $U$ on $\omega_1$. 
 For each $q\in D$ below $p$, we can now find $\alpha_q\in d_q\setminus d_p$ and $i_q<n$ with the property that $$A_q ~ = ~ \Set{\lambda<\omega_1}{a_q\subseteq\lambda\times\lambda, ~ \alpha^\lambda_q=\alpha_q , ~ f^\lambda_q=f^\lambda_{i_q}} ~ \in ~ U.$$
 Then, for all $q,r\in D$ below $p$ with $\alpha_q=\alpha_r$ and $i_q=i_r$, we can find $\lambda\in A_q\cap A_r$ and this allows us to conclude that $$a_q(\alpha_q) ~ = ~ f^\lambda_{i_q}(\alpha_q) ~ = ~ a_r(\alpha_q).$$
 This shows that there are functions  $\map{h_0,\ldots,h_{n-1}}{\omega_1}{\omega_1}$ such that $h_{i_q}(\alpha_q)=a_q(\alpha_q)$ holds for all $q\in D$ below $p$. %
 Define $$q ~ = ~ \langle a_p, ~ \mathscr{F}_p\cup\Set{h_i}{i<n}, ~ t_p\rangle.$$ Then it is easy to see that $q$ is a condition in $\EEE$ below $p$. Pick $r\in D$ below $q$. 
 Then $\alpha_r\in d_r\setminus d_p=d_r\setminus d_q$ and $a_r(\alpha_r)=h_{i_r}(\alpha_r)$, contradicting the fact that $r\leq_\EEE q$.  
\end{proof}

\begin{lemma}\label{lemma:GenericCondition}
 Let $\theta$ be a sufficiently large regular cardinal and let $Z$ be a countable elementary submodel of $\HH{\theta}$.
 If $p$ is a condition in $\EEE$ with  $\HH{\omega_2}\cap Z\in t_p(N)$ for some $N\in\mathscr{C}_p$, then $p$ is a $(Z,\EEE)$-generic condition. 
\end{lemma}

\begin{proof}
 Pick a dense subset $D$ of $\EEE$ that is contained in $Z$ and a condition $q$ in $\EEE$ below $p$. 
 Set $\mathscr{C}=\mathscr{C}_q\cap N$. 
 Then $\mathscr{C}$ is a finite $\in$-chain of countable transitive sets. 
 Moreover, the definition of our forcing ensures that for every $M\in\mathscr{C}$, we can find $X\in t_q(M)$ and $Y\in t_q(N)$ with $X\in Y$. In particular, we know that every element of $\mathscr{C}$ is countable  in $N$ and this shows that $\mathscr{C}$ is a subset of $Z$.

  Now, let $\betrag{t_q(N)}=k>0$ and pick sets $Z_0,\ldots,Z_{k-1}$ such that $Z_0=\HH{\omega_2}\cap Z$ and $t_q(N)=\{Z_0,\ldots,Z_{k-1}\}$. 
 Given $i<k$, since $N$ is the transitive collapse of $Z_i$, there exists a unique isomorphism $\map{\pi_i}{\langle Z_i,\in\rangle}{\langle Z_0,\in\rangle}$. 
 Next, for all  $M\in\mathscr{C}$ and $X\in t_q(M)$, we let $I(X)$ denote the set of all $i<k$ with the property that $X\in Z_i$ and there exists a finite $\in$-chain $C$ of elements of $Z_i$ with the property that for all $M^\prime\in\mathscr{C}$ with $M\in M^\prime$, there exists $X^\prime\in C$ with $X\in X^\prime\in t_q(M^\prime)$.  
 Note that, since $\mathscr{C}$
 is finite, the definition of $\EEE$ then ensures that $I(X)\neq\emptyset$ holds for all $M\in\mathscr{C}$ and $X\in t_q(M)$. 
 Define $t$ to be the unique function with domain $\mathscr{C}$ and $$t(M) ~ = ~ \Set{\pi_i(X)}{X\in t_q(M), ~ i\in I(X)}$$ for all $M\in\mathscr{C}$. 
 Finally, set $$\bar{q} ~ = ~ \langle a_q\cap Z, ~  \mathscr{F}_q\cap Z, ~ t\rangle.$$

 \begin{claim*}
  $\bar{q}$ is a condition in $\EEE$ that is an element of $Z$. 
 \end{claim*}
 
 \begin{proof}[Proof of the Claim]
  First,  fix $M\in\mathscr{C}$ and $X\in t(M)$. Then there is $X^\prime\in t_q(M)$ and $i\in I(X^\prime)$ with $X=\pi_i(X^\prime)$. 
  Since $X^\prime\subseteq Z_i$ and both sets are elementary submodels of $\HH{\omega_2}$, we know that $X^\prime$ is an elementary submodel of $Z_i$ and therefore elementarity implies that $X$ is  an elementary submodel of both $Z_0$ and $\HH{\omega_2}$. 
  Moreover, since $M$ is the transitive collapse of $X^\prime$, we can conclude that $M$ is also the transitive collapse of $X$.

  Now, fix $M_0,M_1\in\mathscr{C}$ with $M_0\in M_1$ and $X_0\in t(M_0)$. Then there exists $X^\prime\in t_q(M_0)$ and $i\in I(X^\prime)$ with $X=\pi_i(X^\prime)$. 
  By the definition of $I(X^\prime)$, there exists a finite $\in$-chain $C$ of elements of $Z_i$ with the property that for all $M\in\mathscr{C}$ with $M_0\in M$, there exists $X\in C$ with $X^\prime\in X\in t_q(M)$. 
  Pick $X^{\prime\prime}\in C\cap t_q(M_1)$ and set $X_1=\pi_i(X^{\prime\prime})$. 
  Then the $\in$-chain $\Set{X\in C}{X^{\prime\prime}\in X}$ witnesses that $i\in I(X^{\prime\prime})$ and hence $X_1$ is an element of $t(M_1)$ with $X_0\in X_1$.

  Next, pick $M\in\mathscr{C}$, $X\in t(M)$ and $\alpha\in d_q\cap X\cap Z=d_q\cap X$. 
  Then there is $X^\prime\in t_q(M)$ and $i\in I(X^\prime)$ with $X=\pi_i(X^\prime)$. 
  Since $d_q\cap X=d_q\cap X^\prime$, we have $M\in\mathscr{C}_q$, $X^\prime\in t_q(M)$ and $\alpha\in d_q\cap X^\prime$. 
  By the definition of $\EEE$, this implies that $a_q(\alpha)\in X^\prime$ and, since $X\cap\omega_1=X^\prime\cap\omega_1$, we can conclude that $a_q(\alpha)\in X$.

  Finally, fix $M\in\mathscr{C}$, $X\in t(M)$, $\alpha\in (d_q\cap Z)\setminus X$ and a function $f$ from $\omega_1$ to $\omega_1$ in $X$. 
  Pick $X^\prime\in t_q(M)$ and $i\in I(X^\prime)$ with $X=\pi_i(X^\prime)$. 
  Then the fact that $X\cap\omega_1=X^\prime\cap\omega_1$ implies that $\alpha\in d_q\setminus X^\prime$. 
  In this situation, the definition of $\EEE$ and the fact that $\pi_i^{{-}1}\restriction(X\cap\omega_1)=\id_{X\cap\omega_1}$ imply that $$a_q(\alpha) ~ \neq ~ (\pi_i^{{-}1}(f))(\alpha) ~ = ~ \pi_i^{{-1}}(f(\alpha)) ~ = ~ f(\alpha).$$

  The above computations shows that $\bar{q}$ is a condition in $\EEE$. Since all relevant sets are finite, the fact that $\mathscr{C}$ is a subset of $Z$ allows us to conclude that $\bar{q}$  is an element of $Z$. 
 \end{proof}

 An application of Lemma \ref{lemma:LarsonUltraSort} in $Z$ now yields an ordinal $\lambda\in Z\cap\omega_1$ with the property that for every finite set $\mathscr{F}$ of functions from $\omega_1$ to $\omega_1$, there exists $r\in D$ below $\bar{q}$ with $a_r\subseteq\lambda\times\lambda$ and $a_r(\alpha)\neq f(\alpha)$ for all $\alpha\in d_r\setminus d_{\bar{q}}$ and $f\in\mathscr{F}$. 
  Hence, there exists $r\in D$ below $\bar{q}$ with $a_r\subseteq\lambda\times\lambda$ and $a_r(\alpha)\neq f(\alpha)$ for all $\alpha\in d_r\setminus d_{\bar{q}}$ and $f\in\mathscr{F}_q$.
  Since $a_r\subseteq\lambda\times\lambda\subseteq Z$, elementarity yields a condition $s\in D\cap Z$ with $a_r=a_s$ and $s\leq_\EEE \bar{q}$. 
  Define $\mathscr{C}_*=\mathscr{C}_q\cup\mathscr{C}_s$ and let  $t_*$ denote the unique function with domain $\mathscr{C}_*$ such that $t_*(M)=t_q(M)$ for all $M\in\mathscr{C}_q\setminus\mathscr{C}_s$ and $$t_*(M) ~ = ~  \Set{\pi_i^{{-}1}(X)}{i<k, ~ X\in t_s(M)}$$ 
  for all $M\in\mathscr{C}_s$. 
  Finally, we set  $$u ~ = ~ \langle a_q\cup a_s, \mathscr{F}_q\cup\mathscr{F}_s,t_*\rangle.$$

  \begin{claim*}
   $u$ is a condition in $\EEE$. 
  \end{claim*}
  
  \begin{proof}[Proof of the Claim]
   First, fix $\alpha\in d_q\cap d_s$. Since $d_s=d_r\subseteq \lambda\subseteq Z$, we  know that $\alpha\in d_q\cap Z=d_{\bar{q}}$ and therefore the fact that $r\leq_\EEE\bar{q}$ allows us to conclude that $$a_q(\alpha) ~ = ~ a_{\bar{q}}(\alpha) ~ = ~ a_r(\alpha) ~ = ~ a_s(\alpha).$$ In particular, we know that $a_q\cup a_s$ is a function.

   Now, fix $M_0\in\mathscr{C}_q\setminus\mathscr{C}_s$ and $M_1\in\mathscr{C}_s\setminus\mathscr{C}_q$. 
   Then $M_0\notin\mathscr{C}_q\cap N=\mathscr{C}_{\bar{q}}\subseteq\mathscr{C}_s$ and hence $M_0\notin N$. 
   Since $M_0$ and $N$ are both contained in the $\in$-chain $\mathscr{C}_q$, we now know that either $M_0=N$ or $N\in M_0$. 
   But $M_1\in\mathscr{C}_s\subseteq Z$ implies that $M_1\in N$ and therefore we know that $M_1\in M_0$ holds in both cases.  
   These computations show that $\mathscr{C}_*$ is an $\in$-chain.

   Next, pick $M_0,M_1\in\mathscr{C}_*$ with $M_0\in M_1$ and $X_0\in t_*(M_0)$. 
   If $M_0,M_1\in\mathscr{C}_q\setminus\mathscr{C}_s$, then $X_0=t_*(M_0)=t_q(M_0)$ and there is $X_1\in t_q(M_1)=t_*(M_1)$ with $X_0\in X_1$. 
   Now, assume that $M_1\in\mathscr{C}_s$. Since $\mathscr{C}_s\subseteq Z$, we then have $M_1\in N$ and, since $\mathscr{C}_q\cap N=\mathscr{C}_{\bar{q}}\subseteq\mathscr{C}_s$, we know that $M_0\in\mathscr{C}_s$.  
   We can now find $i<k$ and $X^\prime\in t_s(M_0)$ with $X_0=\pi_i^{{-}1}(X^\prime)$. Pick $X^{\prime\prime}\in t_s(M_1)$ with $X^\prime\in X^{\prime\prime}$ and set $X_1=\pi_i^{{-}1}(X^{\prime\prime})$. Then $X_0\in X_1\in t_*(M_1)$. 
   Finally, assume that $M_0\in\mathscr{C}_s$ and $M_1\in\mathscr{C}_q\setminus\mathscr{C}_s$. As above, we  know that either $M_1=N$ or $N\in M_1$. 
   In the first case, if $M_1=N$ and $X_0=\pi_i^{{-}1}(X)$ with $i<k$ and $X\in t_s(M)$, then $X_0\in Z_i\in t_q(N)=t_*(M_1)$. 
   In the other case, if $M_1\in N$ and $X_0=\pi_i^{{-}1}(X)$ with $i<k$ and $X\in t_s(M)$, then $Z_i\in t_q(N)$, there is $X_1\in t_q(M_1)=t_*(M_1)$ with $Z_i\in X_1$ and elementarity implies that $X_0\in X_1$. 
   These computations show that, in all cases, there exists $X_1\in t_*(M_1)$ with $X_0\in X_1$.

   Now, fix $M\in\mathscr{C}_*$, $X\in t_*(M)$ and $\alpha\in (d_q\cup d_s)\cap X$. 
   First, assume that  $M\in\mathscr{C}_q\setminus\mathscr{C}_s$ and $\alpha\in d_s\cap X$. Since $M,N\in\mathscr{C}_q$ and $M\notin\mathscr{C}_q\cap N=\mathscr{C}_{\bar{q}}\subseteq\mathscr{C}_s$, we know that either $M=N$ or $N\in N$, and  both cases imply that $Z\cap\omega_1\subseteq X$. In particular, we know that $(a_q\cup a_s)(\alpha)=a_s(\alpha)\in Z\cap\omega_1\subseteq X$. 
   Next, if $M\in\mathscr{C}_q\setminus\mathscr{C}_s$ and $\alpha\in d_q\cap X$, then $(a_q\cup a_s)(\alpha)=a_q(\alpha)\in X$. 
   Finally, assume that  $M\in\mathscr{C}_s$.
   Note that $d_q\cap X=d_{\bar{q}}\cap X\subseteq d_s\cap X$. In particular, we know that $\alpha\in d_s\cap X$. 
   Pick $i<k$ and $X^\prime\in t_s(M)$ with $X=\pi_i^{{-}1}(X^\prime)$. 
   Then $\alpha\in d_s\cap X^\prime$ and this implies that $a_s(\alpha)\in X^\prime$. But this shows that $(a_q\cup a_s)(\alpha)=a_s(\alpha)\in X$. %
   These computations show that $(a_q\cup a_s)(\alpha)\in X$.

   Finally, assume that $M\in\mathscr{C}_*$, $X\in t_*(M)$, $\alpha\in(d_q\cup d_s)\setminus X$ and $f$ is a function from $\omega_1$ to $\omega_1$ in $X$. 
   First, assume that $M\in\mathscr{C}_q\setminus\mathscr{C}_s$. As above, we then know that either $M=N$ or $N\in M$ holds and, in both cases, we can conclude that $d_s\subseteq Z\cap\omega_1\subseteq X$. This shows that $\alpha\in d_q\setminus X$ and hence $X\in t_q(M)$ implies that $(a_q\cup a_s)(\alpha)=a_q(\alpha)\neq f(\alpha)$. 
   Next, assume that $M\in\mathscr{C}_s$ and $\alpha\in d_q\setminus(d_s\cup X)$. Since $d_q\cap Z=d_{\bar{q}}\subseteq d_s$, we now know that $\alpha\notin Z$ and therefore $f\in X\subseteq Z_0\in t_q(N)$ implies that $(a_q\cup a_s)(\alpha)=a_q(\alpha)\neq f(\alpha)$. 
   Finally, assume that $M\in\mathscr{C}_s$ and $\alpha\in d_s\setminus X$. Pick $i<k$ and $X^\prime\in t_s(M)$ with $X=\pi_i^{{-}1}(X^\prime)$. Then $\alpha\in d_s\setminus X^\prime$ and $$(a_q\cup a_s)(\alpha) ~ = ~ a_s(\alpha) ~ \neq ~  \pi_i^{{-}1}(f)(\alpha) ~ = ~ f(\alpha).$$ These computations show that $(a_q\cup a_s)(\alpha)\neq f(\alpha)$ holds in all cases.  
  \end{proof}

 \begin{claim*}
  $u\leq_\EEE q$. 
 \end{claim*}
 
 \begin{proof}[Proof of the Claim]
  The definition of $u$ directly implies that $d_q\subseteq d_u$, $a_q=a_u\restriction d_q$, $\mathscr{F}_q\subseteq\mathscr{F}_u$ and  $\mathscr{C}_q\subseteq\mathscr{C}_u$. 
  Now, assume that $M\in \mathscr{C}_q\cap\mathscr{C}_s$ and $X\in t_q(M)$. Then $M\in Z$ and therefore $M\in\mathscr{C}_q\cap N=\mathscr{C}_{\bar{q}}$. Fix $i\in I(X)$. Then $\pi_i(X)\in t_{\bar{q}}(M)\subseteq t_s(M)$ and therefore $X\in t_u(M)$. Since we also have $t_u\restriction(\mathscr{C}_q\setminus\mathscr{C}_s)=t_q\restriction(\mathscr{C}_q\setminus\mathscr{C}_s)$, we can conclude that $t_q(M)\subseteq t_u(M)$ holds for all $M\in\mathscr{C}_q$. 
  Finally, fix $\alpha\in d_u\setminus d_q$ and $f\in\mathscr{F}_q$. 
  Then $d_s\subseteq Z$ implies that $\alpha\in d_s\setminus d_q=d_r\setminus d_{\bar{q}}$ and therefore $a_u(\alpha)=a_s(\alpha)=a_r(\alpha)\neq f(\alpha)$.  
 \end{proof}

 \begin{claim*}
  $u\leq_\EEE s$. 
 \end{claim*}
 
 \begin{proof}[Proof of the Claim]
  The definition of $u$ together with the fact that $\pi_0=\id_{Z_0}$ directly imply that $d_s\subseteq d_u$, $a_s=a_u\restriction d_s$, $\mathscr{F}_s\subseteq\mathscr{F}_u$, $\mathscr{C}_s\subseteq\mathscr{C}_u$ and $t_s(M)\subseteq t_u(M)$ for all $M\in\mathscr{C}_s$. 
  Fix $\alpha\in d_u\setminus d_s$ and $f\in\mathscr{F}_s$. 
  Since $d_q\cap N\subseteq d_s$, we then know that $\alpha\in d_q\setminus Z_0$ and, since $N\in\mathscr{C}_q$, $Z_0\in t_q(N)$ and $f\in Z_0$, we can conclude that $a_u(\alpha)=a_q(\alpha)\neq f(\alpha)$.  
 \end{proof}

 Since $s$ is an element of $D\cap Z$, the above claims show that $p$ is a $(Z,\EEE)$-generic  condition. 
\end{proof}

\begin{corollary}\label{corollary:Proper}
  The partial order $\EEE$ is proper.
\end{corollary}

\begin{proof}
 Fix a sufficiently large regular cardinal $\theta$, a countable elementary submodel $X$ of $\HH{\theta}$ and a condition $p$ in $\EEE$ that is an element of $X$. 
 Let $M$ denote the transitive collapse of $\HH{\omega_2}\cap X$ and define $$q ~ = ~ \langle a_p,\mathscr{F}_p,t_p\cup\{\langle M,\{\HH{\omega_2}\cap X\}\rangle\}\rangle.$$ 
 Then it is easy to see that $q$ is a condition in $\EEE$ below $p$. 
 Moreover, Lemma \ref{lemma:GenericCondition} shows that $q$ is $(X,\EEE)$-generic. 
\end{proof}

Following {\cite{MR675955}} and {\cite{MR792822}}, we give the following definition.
\begin{definition}

A partial order $\PPP$ satisfies the \emph{$\aleph_2$-isomorphism condition}\footnote{In \cite{MR675955}, this property is called \emph{$\aleph_2$-properness isomorphism condition}. We follow the naming conventions of \cite{MR792822}.} if  for 
 \begin{itemize}
  \item all sufficiently large  regular cardinals $\theta$, 
  
  \item all well-orderings $\lhd$ of $\HH{\theta}$, 
  
  \item all ordinals  $\alpha<\beta<\omega_2$, 
  
  \item all countable elementary submodel $Y$ and $Z$ of $\langle\HH{\theta},\in,\lhd\rangle$ with $\alpha\in Y$, $\beta\in Z$ and $\PPP\in Y\cap Z$,  $Y\cap\omega_2\subseteq \beta$ and $Y\cap \alpha=Z\cap\beta$,
  
  \item all conditions $p$ in $\PPP$ that are contained in $Y$, and 
  
  \item all isomorphisms $\map{\pi}{\langle Y,\in\rangle}{\langle Z,\in\rangle}$ with $\pi(\alpha)=\beta$ and $\pi\restriction(Y\cap Z)=\id_{Y\cap Z}$, 
 \end{itemize}   
 there exists an $(Y,\PPP)$-generic condition $q$ below both $p$ and $\pi(p)$ with the property that $\pi[G\cap Y]=G\cap Z$ holds whenever $G$ is $\PPP$-generic over $\VV$ with $q\in G$. 
\end{definition}

\begin{lemma}
\label{lemma:IsoCondition} The partial order $\EEE$ satisfies the $\aleph_2$-isomorphism condition. 
\end{lemma}

\begin{proof}
 In the following, pick $\theta$, $\lhd$, $\alpha$, $\beta$, $Y$, $Z$, $p$ and $\pi$ as in the definition of the $\aleph_2$-isomorphism condition. %
 Then  it is easy to see that $\pi(p)$ is again a condition in $\EEE$ with  $d_p=d_{\pi(p)}$, $a_p=a_{\pi(p)}$  and $\mathscr{C}_p=\mathscr{C}_{\pi(p)}$. 
 Let $t$ denote the unique function with domain $\mathscr{C}_p$ and $t(M)=t_p(M)\cup t_{\pi(p)}(M)$ for all $M\in\mathscr{C}_p$. 
 Then it is easy to see that the tuple $$q ~ = ~ \langle a_p, ~ \mathscr{F}_p\cup\mathscr{F}_{\pi(p)}, ~ t\rangle$$ is a condition in $\EEE$ below both $p$ and $\pi(p)$.

 Now, let $N$ denote the transitive collapse of $\HH{\omega_2}\cap Y$ and  define $$r ~ = ~ \langle a_q, ~ \mathscr{F}_q, ~ t_q\cup\{\langle N,\{\HH{\omega_2}\cap Y,\HH{\omega_2}\cap Z\}\rangle\}.$$
 Since our assumptions imply that $Y\cap\omega_1=Z\cap\omega_1$ and $\pi\restriction(Y\cap\omega_1)=\id_{Z\cap\omega_1}$, it follows that $r$ is a condition in $\EEE$ below $q$. Moreover, Lemma \ref{lemma:GenericCondition} directly implies that $r$ is both an $(Y,\EEE)$- and a $(Z,\EEE)$-generic condition.

 In the following, let $G$ be $\EEE$-generic over $\VV$ with $r\in G$. 
 Assume, towards a contradiction, that there is $s\in G\cap Y$ with $\pi(s)\notin G$. 
 Fix a condition $u$ in $\EEE$ below both $r$ and $s$. 
 Set $k=\betrag{t_u(N)}>1$ and pick sets $W_0,\ldots,W_{k-1}$ with $W_0=\HH{\omega_2}\cap Y$, $W_1=\HH{\omega_2}\cap Z$ and $t_u(N)=\{W_0,\ldots,W_{k-1}\}$. 
 For all $M\in\mathscr{C}_u\cap N$ and all $X\in t_u(M)$, we let $I(X)$ denote the set of all $i<k$ with the property that $X\in W_i$ and there exists a  finite $\in$-chain $C$ of elements of $W_i$ with the property that for all $M\in M^\prime\in C$, there exists $X\in X^\prime\in C\cap t_u(M^\prime)$. Then $I(X)\neq\emptyset$ for all $M\in\mathscr{C}_u\cap N$ and all $X\in t_u(M)$. 
 Next, given $i,j<k$, let $\map{\pi_{i,j}}{\langle W_i,\in\rangle}{\langle W_j,\in\rangle}$ denote the unique isomorphism between these structures. We then have  $\pi_{0,1}=\pi\restriction(W_0)$. %
 We now define $t_*$ to be the unique function with domain $\mathscr{C}_u$ such that $t_*(M)=t_u(M)$ holds for all $M\in\mathscr{C}_u\setminus N$ and $$t_*(M) ~ = ~ \Set{\pi_{i,j}(X)}{X\in t_u(M), ~ i\in I(X), ~ j<k}$$ for all $M\in\mathscr{C}_u\cap N$. Set $$v ~ = ~ \langle a_u, ~ \mathscr{F}_{\pi(s)}\cup\mathscr{F}_u, ~ t_*\rangle.$$

 \begin{claim*}
  $v$ is a condition in $\EEE$ below $u$. 
 \end{claim*}
 
 \begin{proof}[Proof of the Claim]
  First, if $M\in\mathscr{C}_u\cap N$, $X\in t_u(M)$, $i\in I(X)$ and $j<k$, then $X$ is an elementary submodel of $W_i$ and this allows us to conclude that $\pi_{i,j}(X)$ is a countable elementary submodel of $\HH{\omega_2}$ whose transitive collapse is equal to $M$. 
  
  Next, fix $M_0,M_1\in\mathscr{C}_u$ with $M_0\in M_1$ and $X_0\in t_*(M_0)$. 
  First, assume that $M_1\in N$. Then $M_0\in N$ and we can find $X\in t_u(M_0)$, $i\in I(X)$ and $j<k$ with $X_0=\pi_{i,j}(X)$. 
  Let $C$ be the $\in$-chain of elements of $W_i$ witnessing that $i\in I(X)$ and pick $X^\prime\in C$ with $X\in X^\prime\in C\cap t_u(M_1)$. Then $\Set{W\in C}{X^\prime\in W}$ is an $\in$-chain witnessing that $i\in I(X^\prime)$ and hence we have $X_0\in \pi_{i,j}(X^\prime)\in t_*(M_1)$. 
  Next, if $M_0,M_1\notin N$, then $X_0\in t_u(M_0)$ and there exists $X_1\in t_r(M_1)=t_*(M_1)$ with $X_0\in X_1$. 
  Finally, assume that  $M_0\in N$ and $M_1\notin N$. 
  Then there exits $j<k$ with $X_0\in W_j\in t_u(N)$. 
  Since $M_1$ and $N$ are both contained in $\mathscr{C}_u$, we then know that either $M_1=N$ or $N\in M_1$. 
  If $N\in M_1$, then we can find $X_1\in t_u(M_1)$ with $W_j\in X_1$ and we then also have $X_0\in X_1$. We can therefore conclude that, in all cases, there exists $X_1\in t_*(M_1)$ with $X_0\in X_1$. 
  
  We now fix $M\in\mathscr{C}_u$, $X\in t_*(M)$ and $\alpha\in d_u\cap X$. 
  If $M\notin N$, then $X\in t_u(M)$ and therefore $a_u(\alpha)\in X$. 
  Now, assume that $M\in N$. Pick $X^\prime\in t_u(M)$, $i\in I(X^\prime)$ and $j<k$ with $X=\pi_{i,j}(X^\prime)$. 
  Then $\alpha\in d_u\cap X^\prime$ and therefore $a_u(\alpha)\in X^\prime\cap\omega_1\subseteq X$. 
  
  Next, fix $M\in\mathscr{C}_u$, $X\in t_*(M)$, $\alpha\in d_u\setminus X$ and a function $\map{f}{\omega_1}{\omega_1}$ in $X$. 
  If $M\notin N$, then $X\in t_u(M)$ and $a_u(\alpha)\neq f(\alpha)$. 
  In the other case, if $M\in N$ and $X=\pi_{i,j}(X^\prime)$ for some $X^\prime\in t_u(M)$, $i\in I(X^\prime)$ and $j<k$, then $\alpha\in d_u\setminus X^\prime$ and therefore $a_u(\alpha)\neq\pi_{i,j}^{{-}1}(f)(\alpha)=f(\alpha)$. 
  
  The above computations show that $v$ is a condition in $\EEE$ with $a_u=a_v$, $\mathscr{F}_u\subseteq\mathscr{F}_v$ and $\mathscr{C}_u=\mathscr{C}_v$. 
  Since our construction ensures that  $t_u(M)\subseteq t_v(M)$ holds for all $M\in\mathscr{C}_u$, we can now conclude that $v\leq_\EEE u$ holds. 
 \end{proof}

  \begin{claim*}
   $v\leq_\EEE \pi(s)$. 
  \end{claim*}
  
  \begin{proof}[Proof of the Claim]
   The fact that $v\leq_\EEE s$ directly implies that  $d_{\pi(s)}=d_s\subseteq d_v$, $a_v\restriction d_{\pi(s)}=a_s=a_{\pi(s)}$ and $\mathscr{C}_{\pi(s)}=\mathscr{C}_s\subseteq\mathscr{C}_v$. 
   Moreover, the definition of $v$ ensures that $\mathscr{F}_{\pi(s)}\subseteq \mathscr{F}_v$. 
   In addition, the fact that $\pi_{0,1}=\pi\restriction W_0$ and $\mathscr{C}_s\subseteq\mathscr{C}_u\cap N$ directly implies that $$t_{\pi(s)}(M) ~ = ~ \pi[t_s(M)] ~ \subseteq ~  \pi[\Set{\pi_{0,1}(X)}{X\in t_u(M),  ~ X\in W_0}] ~ \subseteq ~ t_v(M)$$ holds for all $M\in\mathscr{C}_{\pi(s)}$. 
   Finally, fix $\alpha\in d_v\setminus d_{\pi(s)}$ and $f\in\mathscr{F}_{\pi(s)}$. 
   Then $\alpha\in d_v\setminus d_s$ and $\pi^{{-}1}(f)\in\mathscr{F}_s$. Since this allows us to conclude that $$a_{\pi(s)}(\alpha) ~ = ~ a_s(\alpha) ~ \neq ~  \pi^{{-}1}(f)(\alpha) ~ = ~ f(\alpha),$$ the statement of the claim follows. 
  \end{proof}

  A density argument now shows that there is a condition $v$ in $G$ that is stronger than $\pi(s)$, contradicting our assumption. 
  This shows that $\pi[G\cap Y]\subseteq G$. 
  
  Finally, assume, towards a contradiction, that there is $u\in G\cap Z$ with the property that $\pi^{{-}1}(u)\notin G$. 
  Let $D$ denote the set of all conditions in $\EEE$ that are either stronger than $\pi^{{-}1}(u)$ or incompatible with $\pi^{{-}1}(u)$. 
  Then $D$ is a dense subset of $\EEE$ that is contained in $Y$. Since $r\in G$, we can find $v\in D\cap G\cap Y$. 
  In this situation, the above computations show that $\pi(v)\in G$ and elementarity implies that the element $u$ and $\pi(v)$ are incompatible in $\EEE$, a contradiction.  
 \end{proof}

 The statements of the following proposition can be proven in the same way as the corresponding results for the partial order $\DDD$ in Section \ref{diagonal}. The details are left to the reader.

\begin{proposition}\label{proposition:PropertiesOfEEE}
 Let $G$ be $\EEE$-generic over $\VV$ and set $g=\bigcup\Set{a_p}{p\in G}$. 
 \begin{enumerate}
     \item The set $g$ is a function from $\omega_1$ to $\omega_1$.
     
     \item For every sequence $\seq{f_\alpha}{\alpha<\omega_1}$ of functions from $\omega_1$ to $\omega_1$ in $\VV$, every monotone map $\map{m}{[\omega_1]^{{<}\omega}}{[\omega_1]^{{<}\omega}}$ in $\VV$ and every $a\in[\omega_1]^{{<}\omega}$, there exists $a\subseteq b\in[\omega_1]^{{<}\omega}$ with the property that   $$\Set{\beta<\omega_1}{f_\alpha(\beta)=g(\beta)} ~ \subseteq ~ m(b)$$ holds for all $\alpha\in  m(b)$. 
     
     \item Given  a finite subset $F$ of $\omega_1$,  the tuple $$q_F ~ = ~ \langle\emptyset, ~ \Set{c_\beta}{\beta\in F}, ~ \emptyset\rangle$$ is a condition in $\EEE$ iin $\VV$ and, if $q_F\in G$, then $\ran{g}\cap F=\emptyset$. \qed 
 \end{enumerate}
\end{proposition}


Before we give the proof of the  main result of this section, Theorem \ref{theorem:ConZFC}, we state the following direct consequence of {\cite[Section VIII, Lemma 2.4]{MR675955}} and the iteration theorem for proper forcings that is used in our construction.

\begin{lemma} \label{lemma:shelah}
 Let $$\langle\seq{\vec{\PPP}_{{<}\gamma}}{\gamma\leq\omega_2}, ~ \seq{\dot{\PPP}_\gamma}{\gamma<\omega_2} \rangle$$ denote a forcing iteration with countable support. 
 If $\CH$ holds and $$\mathbbm{1}_{\vec{\PPP}_{{<}\gamma}}\Vdash\anf{\textit{$\dot{\PPP}_\gamma$ is proper and satisfies the $\aleph_2$-isomorphism condition}}$$ holds for all $\gamma<\omega_2$, then $\vec{\PPP}_{{<}\omega_2}$ satisfies the $\aleph_2$-chain condition.  \qed
\end{lemma}

We are now ready to give the proof of Theorem \ref{theorem:ConZFC}.
\begin{proof}[Proof of Theorem \ref{theorem:ConZFC}]  
 Assume that $\CH$ holds and fix an enumeration $\seq{F_\gamma}{\gamma<\omega_2}$ of all finite subsets of $\omega_1$ with the property that every such subset is enumerated unboundedly often in $\omega_2$. 
 Let $$\langle\seq{\vec{\PPP}_{{<}\gamma}}{\gamma\leq\omega_2}, ~ \seq{\dot{\PPP}_\gamma}{\gamma<\omega_2} \rangle$$ denote a forcing iteration with countable support with the property that for all $\gamma<\omega_2$, if $G$ is $\vec{\PPP}_{{<}\gamma}$-generic over $\VV$, then $\dot{\PPP}_\gamma^G$ is equal to the suborder of $\EEE^{\VV[G]}$ consisting of all conditions below $q_{F_\gamma}$. 
 Then Corollary \ref{corollary:Proper} implies that $\vec{\PPP}_{{<}\omega_2}$ is proper. 
 In addition, Lemma \ref{lemma:IsoCondition} allows us to apply Lemma  \ref{lemma:shelah} to show that $\vec{\PPP}_{{<}\omega_2}$ satisfies the $\aleph_2$-chain condition. 
 These arguments show that forcing with $\vec{\PPP}_{{<}\omega_2}$ preserves both $\omega_1$ and $\omega_2$. 
 In addition, we know that every subset of $\omega_1$ in a $\vec{\PPP}_{{<}\omega_2}$-generic extension is contained in a proper intermediate extension of the iteration. 
 
 Now, let $G$ be $\vec{\PPP}_{{<}\omega_2}$-generic over $\VV$ and, in $\VV[G]$, fix a sequence $\vec{f}=\seq{f_\alpha}{\alpha<\omega_1}$ of functions from $\omega_1$ to $\omega_1$, a finite subset $F$ of $\omega_1$ and some monotone map $\map{m}{[\omega_1]^{{<}\omega}}{[\omega_1]^{{<}\omega}}$. 
 By the definition of our iteration and the above remarks, there exists $\gamma<\omega_2$ with the property that, if $\bar{G}$ denotes the filter on $\vec{\PPP}_{{<}\gamma}$ induced by $G$, then $F_\gamma=F$ and $\vec{f},m\in\VV[\bar{G}]$. 
 Let $G_\gamma$ be the filter on $\dot{\PPP}_\gamma$ induced by $G$ and set $g=\bigcup\Set{a_p}{p\in G_\gamma}$. 
 Then Proposition \ref{proposition:PropertiesOfEEE} shows that $g$ is a function from $\omega_1$ to $\omega_1$ with the property that $F\cap\ran{g}=\emptyset$ and, for all $a\in[\omega_1]^{{<}\omega}$, there exists $a\subseteq b\in[\omega_1]^{{<}\omega}$ with $\Set{\beta<\omega_1}{f_\alpha(\beta)=g(\beta)}\subseteq m(b)$ for all $\alpha\in m(b)$. 
 This shows that $(\diag)$ holds in $\VV[G]$. 
\end{proof}


\section{The coloring}\label{coloring}

 We now use the principle $(\diag)$ to construct an $(\M,\N)$-full structure of cardinality $\aleph_2$. 
 This implication is an immediate consequence of the next result.

\begin{theorem}\label{main2}
 Assume that $(\diag)$ holds. Then there exists 
 \begin{itemize}
     \item a map $\map{c}{[\omega_2]^2}{\omega_1}$, 
     
     \item a monotone map $\map{m}{[\omega_2]^{{<}\omega}}{[\omega_2]^{{<}\omega}}$, and 
     
     \item a map $\map{r}{\omega_2}{\omega_2}$
 \end{itemize} 
 such that the following statements hold: 
 \begin{enumerate}
  \item If $a\in[\omega_2]^{{<}\omega}$ and $\alpha,\beta\in m(a)$ with $\alpha\neq\beta$, then $\calA^c_{\alpha,\beta}\subseteq m(a)$. 
  
  \item\label{item:FinExt} Given 
   \begin{itemize}
       \item a finite subset $d$ of $\omega_2+\omega$, 
       
       \item a function $\map{e}{[d]^2}{\omega_1}$, and 
       
       \item  a function $\map{s}{d}{\omega_2}$ 
   \end{itemize}     
   such that
   \begin{itemize}
       \item $c\restriction[d\cap\omega_2]^2=e\restriction[d\cap\omega_2]^2$, 
       
       \item $r\restriction d=s$,  and  
       
       \item $\calA^e_{\alpha,\beta} 
       \subseteq \omega_2$ for all $\alpha,\beta\in d\cap\omega_2$ with $\alpha\neq\beta$, 
   \end{itemize}    
   there exists an injection $\map{\iota}{d}{\omega_2}$ with \begin{itemize}
       \item $\iota\restriction(d\cap\omega_2)=\id_{d\cap\omega_2}$, 
       
       \item $c(\iota(\alpha),\iota(\beta)) =  e(\alpha,\beta)$ for all $\alpha,\beta\in d$ with $\alpha\neq\beta$, and 
       \item $r(\iota(\alpha))=s(\alpha)$ for all $\alpha\in d$. 
   \end{itemize}
 \end{enumerate}
\end{theorem}

Before we present the proof of the above theorem, we briefly show how it can be applied to prove the desired independence result.

\begin{corollary} Assume that  $(\diag)$ holds and let $\M$ be a countable model that characterizes $\aleph_1$.  Then there exists  an $(\M,\N)$-full structure of cardinality  $\aleph_2$. 
\end{corollary}

\begin{proof}
 Let $c$, $m$ and $r$ be the functions given by Theorem \ref{main2}. 
 Then the function $m$ directly witnesses that the function $c$ possesses the properties \eqref{item:FiniteAgree} and \eqref{item:FiniteClosure} listed in Lemma \ref{EquivalentLemma}. 
 Now, fix a finite set $d$, a function  $\map{e}{[d]^2}{\omega_{\alpha}}$ with $e\restriction[d\cap\omega_2]^2=c\restriction[d\cap\omega_2]^2$ that introduces no new agreements over $c\restriction[d\cap\omega_2]^2$ and a function $\map{s}{d}{\omega_2}$ with $s\restriction(d\cap\omega_2)=r\restriction(d\cap\omega_2)$.  Without loss of generality, we may assume that $d$ is a subset of $\omega_2+\omega$. 
 If $\alpha,\beta\in d\cap\omega_2$ with $\alpha\neq\beta$, then the fact that $e$ introduces no new agreements over $c\restriction[d\cap\omega_2]^2$ implies that $\calA^e_{\alpha,\beta}=\calA^{c\restriction[d\cap\omega_2]^2}_{\alpha,\beta}\subseteq\omega_2$. %
 This allows us to use conclusion  \eqref{item:FinExt} of Theorem \ref{main2} to find an injection $\map{\iota}{d}{\omega_2}$ with $\iota\restriction(d\cap\omega_2)=\id_{d\cap\omega_2}$, $c(\iota(\alpha),\iota(\beta)) =  e(\alpha,\beta)$ for all $\alpha,\beta\in d$ with $\alpha\neq\beta$ and  $r(\iota(\alpha))=s(\alpha)$ for all $\alpha\in d$. These computations allow us to conclude that the functions $c$ and $r$  possess the properties \eqref{item:FiniteExt} and \eqref{item:Color} listed in Lemma \ref{EquivalentLemma}. We can  therefore apply Lemma \ref{EquivalentLemma} to find  an $(\M,\N)$-full structure of cardinality  $\aleph_2$.  
\end{proof}

\begin{proof}[Proof of Theorem \ref{main2}]
 In the following, we let $\map{\goedel{\cdot}{\cdot}}{\On\times\On}{\On}$ denote the \emph{G\"odel pairing function}. 
 In addition, let $\map{p_0,p_1}{\omega_1}{\omega_1}$ denote the corresponding \emph{projections} on $\omega_1$, {i.e.} the unique pair of functions on $\omega_1$ with $\alpha=\goedel{p_0(\alpha)}{p_1(\alpha)}$ for all $\alpha<\omega_1$. 
 For each $0<\alpha<\omega_2$, fix a surjection $\map{s_\alpha}{\omega_1}{\alpha}$. In addition, pick an enumeration $\seq{\langle e_\xi,s_\xi\rangle}{\xi<\omega_2}$ of all pairs $\langle e,s\rangle$ of functions with $\map{e}{[d]^2}{\omega_1}$ and $\map{s}{d}{\omega_2}$ for some  finite subset $d$ of $\omega_2+\omega$ such that the enumeration has the property that for all  $\zeta<\omega_2$, the set $\Set{\xi<\omega_2}{e_\xi=e_\zeta}$ is unbounded in $\omega_2$.

 Given $0<\alpha<\omega_2$, a finite subset $F$ of $\omega_1$, a map $\map{c_0}{[\alpha]^2}{\omega_1}$ and a monotone map $\map{m_0}{[\alpha]^{{<}\omega}}{[\alpha]^{{<}\omega}}$, we call a pair $(c,m)$ an \emph{$F$-good extension} of $(c_0,m_0)$ if there exists a function $\map{g}{\alpha}{\omega_1}$ such that the following statements hold: 
 \begin{itemize}
  \item $F\cap\ran{g}=\emptyset$. 
  
  \item $\map{c}{[\alpha+1]^2}{\omega_1}$ is a map with $c\restriction[\alpha]^2=c_0$ and $$c(\alpha,\beta) ~ = ~ \goedel{g(\beta)}{\min(s_\alpha^{{-}1}\{\beta\})}$$ for all $\beta<\alpha$.
  
  \item $\map{m}{[\alpha+1]^{{<}\omega}}{[\alpha+1]^{{<}\omega}}$ is a map with $m\restriction[\alpha]^{{<}\omega}=m_0$ and the property that for all $a\in[\alpha]^{{<}\omega}$, there exists $a\subseteq b\in[\alpha]^{{<}\omega}$ satisfying $$m(a\cup\{\alpha\}) ~ = ~ m_0(b)\cup\{\alpha\}$$ and $$\Set{\gamma\in\alpha\setminus\{\beta\}}{p_0(c_0(\beta,\gamma))=g(\gamma)} ~ \subseteq ~ m_0(b)$$ for all $\beta\in m_0(b)$. 
 \end{itemize}
 Note that $(\diag)$ implies that an $F$-good extension of $(c_0,m_0)$ exists.

 In the following, we construct 
 \begin{itemize}
  \item a strictly increasing sequence $\seq{\alpha_\xi}{\xi<\omega_1}$ of ordinals less than $\omega_2$ with $\alpha_0=0$, 
  
  \item a sequence $\seq{\map{c_\xi}{[\alpha_\xi]^2}{\omega_1}}{\xi<\omega_2}$ of maps with $c_\xi\restriction[\alpha_\zeta]^2=c_\zeta$ for all $\zeta\leq\xi<\omega_2$,  
  
  \item a sequence $\seq{\map{m_\xi}{[\alpha_\xi]^{{<}\omega}}{[\alpha_\xi]^{{<}\omega}}}{\xi<\omega_2}$ of monotone maps satisfying  $m_\xi\restriction[\alpha_\zeta]^{{<}\omega}=m_\zeta$ for all $\zeta\leq\xi<\omega_2$, and 
  
  \item a sequence $\seq{\map{r_\xi}{\alpha_\xi}{\omega_2}}{\xi<\omega_2}$ of functions with $r_\xi\restriction\alpha_\zeta=r_\zeta$ for all $\zeta\leq\xi<\omega_2$. 
 \end{itemize}

 Fix $0<\xi<\omega_2$ and assume that $\alpha_\zeta$, $c_\zeta$ and $m_\zeta$ with the above properties are defined for all $\zeta<\xi$.  
 We define
 \begin{itemize}
  \item $\alpha_*=\sup_{\zeta<\xi}\alpha_\zeta$,
  
  \item $\map{c_*=\bigcup_{\zeta<\xi}c_\zeta}{[\alpha_*]^2}{\omega_1}$,  
  
  \item $\map{m_*=\bigcup_{\zeta<\xi}m_\zeta}{[\alpha_*]^{{<}\omega}}{[\alpha_*]^{{<}\omega}}$,   
  
  \item $\map{r_*}{\alpha_*}{\omega_2}$, and 
  
  \item $d=\dom{s_\xi}\in[\omega_2+\omega]^{{<}\omega}$. 
 \end{itemize}

  If either $d\subseteq\alpha_*$, or $d\cap\omega_2\nsubseteq\alpha_*$, or $c_*\restriction[d\cap\omega_2]^2\neq e_\xi\restriction[d\cap\omega_2]^2$, or ${r_*\restriction(d\cap\omega_2)}\neq {s_\xi\restriction(d\cap\omega_2)}$, or there exist  $\beta,\gamma\in d\cap\omega_2$ with $\beta\neq\gamma$ and $\calA^{e_\xi}_{\beta,\gamma}\nsubseteq\omega_2$, then we say that $\xi$ has \emph{Type $0$}, we set $\alpha_\xi=\alpha_*+1$, we define $r_\xi=r_*\cup\{\langle\alpha_*,0\rangle\}$, and we pick $c_\xi$ and $m_\xi$ such that  the pair $(c_\xi,m_\xi)$ is a $\emptyset$-good extension of $(c_*,m_*)$. 
  In the following, assume that $d\nsubseteq\alpha_*$, $d\cap\omega_2\subseteq\alpha_*$, $c_*\restriction[d\cap\omega_2]^2=e_\xi\restriction[d\cap\omega_2]^2$, $r_*\restriction(d\cap\omega_2)=s_\xi\restriction(d\cap\omega_2)$ and $\calA^{e_\xi}_{\beta,\gamma}\subseteq\omega_2$ for all $\beta,\gamma\in d\cap\omega_2$ with $\beta\neq\gamma$. 
   We then say that $\xi$ has \emph{Type $1$}. 
   Let $\betrag{d\setminus\omega_2}=n>0$, set $\alpha_\xi=\alpha_*+n$ and pick an injection $\map{\iota_\xi}{d}{\alpha_\xi}$ with $\iota_\xi\restriction(d\cap\omega_2)=\id_{d\cap\omega_2}$ and $\iota_\xi[d\setminus\omega_2]=[\alpha_*,\alpha_\xi)$. 
   In addition, set $F=\ran{p_0\circ e_\xi}$. 
   Pick a map $\map{c^*}{[\alpha_\xi]^2}{\omega_1}$ and a monotone map $\map{m^*}{[\alpha_\xi]^{{<}\omega}}{[\alpha_\xi]^{{<}\omega}}$ such that $c^*\restriction[\alpha_*]^2=c_*$, $m^*\restriction[\alpha_*]^{{<}\omega}=m_*$ and for all $i<n$, the pair $(c^*\restriction[\alpha_*+i+1]^2,m^*\restriction[\alpha_*+i+1]^{{<}\omega})$ is an $F$-good extension of $(c^*\restriction[\alpha_*+i]^2,m^*\restriction[\alpha_*+i]^{{<}\omega})$. 
   Define $\map{c_\xi}{[\alpha_\xi]^2}{\omega_1}$ to be the unique map with $c_\xi(\iota_\xi(\beta),\iota_\xi(\gamma))=e_\xi(\beta,\gamma)$ for all $\beta,\gamma\in d$ with $\beta\neq\gamma$ and $c_\xi(\beta,\gamma)=c^*(\beta,\gamma)$ for all $\beta<\gamma<\alpha_\xi$ with $\{\beta,\gamma\}\nsubseteq\ran{\iota_\xi}$. 
 Our assumptions on $e_\xi$ then ensure that $c_\xi\restriction[\alpha_*]^2=c_*$. 
 In addition, let $\map{m_\xi}{[\alpha_\xi]^{{<}\omega}}{[\alpha_\xi]^{{<}\omega}}$ denote the unique function with $m_\xi\restriction[\alpha_*]^{{<}\omega}=m_*$ and $m_\xi(a)=m^*(a\cup\ran{\iota_\xi})$ for all $a\in[\alpha_\xi]^{{<}\omega}$ with $a\cap[\alpha_*,\alpha_\xi)\neq\emptyset$. 
 Finally, define $r_\xi$ to be the unique function with domain $\alpha_\xi$ satisfying $r_\xi\restriction\alpha_*=r_*$ and $r_\xi(\iota_\xi(\alpha))=s_\xi(\alpha)$ for all $\alpha\in d$.  
   %
   %
   %
  %
  
  We can now define 
  \begin{itemize}
      \item $\map{c  =  \bigcup_{\xi<\omega_2}c_\xi}{[\omega_2]^2}{\omega_1}$, 
      
      \item $\map{m  =  \bigcup_{\xi<\omega_2}m_\xi}{[\omega_2]^{{<}\omega}}{[\omega_2]^{{<}\omega}}$, and 
      
      \item $\map{r=\bigcup_{\xi<\omega_2}r_\xi}{\omega_2}{\omega_2}$. 
  \end{itemize}  
  Our construction then ensures that $m$ is a monotone map.

  \begin{claim*}
   If $0<\xi<\omega_2$, $a\in[\alpha_\xi]^{{<}\omega}$ and $\alpha,\beta\in m(a)$ with $\alpha\neq\beta$, then $\calA^{c_\xi}_{\alpha,\beta}\subseteq m(a)$. 
  \end{claim*}
  
  \begin{proof}[Proof of the Claim]
   Assume that $0<\xi<\omega_2$ has the property  that $\calA^{c_\zeta}_{\alpha,\beta}\subseteq m(a)$ holds for all $0<\zeta<\xi$, every $a\in[\alpha_\zeta]^{{<}\omega}$ and all $\alpha,\beta\in m(a)$ with $\alpha\neq\beta$. 
   Set $\alpha_*=\sup_{\zeta<\xi}\alpha_\zeta$ and $c_*=c\restriction[\alpha_*]^2$.  
   Then our assumptions imply that $\calA^{c_*}_{\beta,\gamma}\subseteq m(a)$ holds for every $a\in[\alpha_*]^{{<}\omega}$ and all $\beta,\gamma\in m(a)$ with $\beta\neq\gamma$.

   First, assume that $\xi$ has Type $0$. Then $\alpha_\xi=\alpha_*+1$. Fix $a\in[\alpha_\xi]^{{<}\omega}$ and let $\map{g}{\alpha_*}{\omega_1}$ denote the function used in the construction of $c_\xi$. 
   If $\gamma<\beta<\alpha_*$, then  $$s_{\alpha_*}(p_1(c_\xi(\alpha_*,\beta)))  ~ = ~  \beta ~ \neq ~ \gamma ~ = ~ s_{\alpha_*}(p_1(c_\xi(\alpha_*,\gamma)))$$  and therefore $c_\xi(\alpha_*,\beta)\neq c_\xi(\alpha_*,\gamma)$. 
  In particular, if $a\subseteq\alpha_*$ and $\beta,\gamma\in m(a)$ with $\beta\neq\gamma$, then $\calA^{c_\xi}_{\beta,\gamma}=\calA^{c_*}_{\beta,\gamma}\subseteq m(a)$. 
  In the following, assume that $\alpha_*\in a$. Then there exists $a\cap\alpha_*\subseteq b\in[\alpha_*]^{{<}\omega}$ with $m(a)=m(b)\cup\{\alpha_*\}$ and $$\Set{\gamma\in\alpha_*\setminus\{\beta\}}{p_0(c_*(\beta,\gamma))=g(\gamma)} ~ \subseteq ~ m(b)$$ for all $\beta\in m(b)$. 
  Now, if $\beta,\gamma\in m(a)\cap\alpha_*=m(b)$ with $\beta\neq\gamma$, then the above computations show that $\calA^{c_\xi}_{\beta,\gamma}=\calA^{c_*}_{\beta,\gamma}\subseteq m(b)\subseteq m(a)$. 
   Moreover, if we have $\beta\in m(a)\cap\alpha_*=m(b)$ and $\gamma\in\alpha_*\setminus\{\beta\}$ satisfying  $c_\xi(\alpha_*,\gamma)=c_\xi(\beta,\gamma)$, then $p_0(c_*(\beta,\gamma))=g(\gamma)$ and therefore $\gamma\in m(b)\subseteq m(a)$. 
   This shows that $\calA^{c_\xi}_{\alpha_*,\beta}\subseteq m(a)$ holds for all $\beta\in m(a)$ with $\alpha_*\neq\beta$.

    Now, assume that $\xi$ has Type $1$. 
    Let $d$ denote the unique finite subset of $\omega_2+\omega$ with the property that the domain of $e_\xi$ is equal to the set $[d]^2$. 
    In addition, let  $\map{c^*}{[\alpha_\xi]^2}{\omega_1}$ and $\map{m^*}{[\alpha_\xi]^{{<}\omega}}{[\alpha_\xi]^{{<}\omega}}$ denote the functions used in the construction of $c_\xi$. 
    Then the above computations show that $\calA^{c^*}_{\beta,\gamma}\subseteq\alpha_*$ holds for all $\gamma<\beta<\alpha_*$ and $\calA^{c^*}_{\beta,\gamma}\subseteq m^*(a)$ holds for every $a\in[\alpha_\xi]^{{<}\omega}$ and all $\beta,\gamma\in m^*(a)$ with $\beta\neq \gamma$.

    \begin{subclaim*}
     If $\gamma<\beta<\alpha_*$, then $\calA^{c_\xi}_{\beta,\gamma}\subseteq\alpha_*$. 
    \end{subclaim*}
    
    \begin{proof}[Proof of the Subclaim]
     Assume, towards a contradiction, that $c_\xi(\beta,\delta)=c_\xi(\gamma,\delta)$ holds for some $\alpha_*\leq\delta<\alpha_\xi$. 
     Then  $\delta\in\ran{\iota_\xi}$ and we know that $\{\beta,\gamma\}\nsubseteq\ran{\iota_\xi}$, because otherwise we would have $\beta,\gamma\in d\cap\omega_2$ and $\calA^{e_\xi}_{\beta,\gamma}\nsubseteq\omega_2$. 
     Now, if $\gamma\notin\ran{\iota_\xi}$, then the fact that $$p_0(c_\xi(\beta,\delta)) ~ = ~ p_0(c_\xi(\gamma,\delta)) ~ = ~ p_0(c^*(\gamma,\delta)) ~ \notin ~ \ran{p_0\circ e_\xi}$$ implies that $\beta\notin\ran{\iota_\xi}$. 
     The same argument shows that $\beta\notin\ran{\iota_\xi}$ implies that $\gamma\notin\ran{\iota_\xi}$. 
     Hence, we can conclude that $\beta$ and $\gamma$ are both not contained in $\ran{\iota_\xi}$. But then our assumption implies that $c^*(\beta,\delta)=c^*(\gamma,\delta)$ and, by the above remarks, this shows that $\delta<\alpha_*$, a contradiction. 
    \end{proof}

    Fix $a\in[\alpha_\xi]^{{<}\omega}$. 
    If $a\subseteq\alpha_*$, then our subclaim shows that $\calA^{c_\xi}_{\beta,\gamma}=\calA^{c_*}_{\beta,\gamma}\subseteq m(a)$ holds for all $\beta,\gamma\in m(a)$ with $\beta\neq\gamma$. 
    In the following, assume that $a\cap[\alpha_*,\alpha_\xi)\neq\emptyset$. Then $m(a)=m^*(a\cup\ran{\iota_\xi})$. 
    Pick $\beta,\gamma\in m(a)$ with $\beta\neq\gamma$ and $\delta\in\calA^{c_\xi}_{\beta,\gamma}\setminus\ran{\iota_\xi}$. 
    Then the definition of $c_\xi$ ensures that $c^*(\beta,\delta)=c^*(\gamma,\delta)$ and hence we know that $\delta\in\calA^{c^*}_{\beta,\gamma}\subseteq m^*(a\cup\ran{\iota_\xi}))=m(a)$ holds. This shows that $$\calA^{c_\xi}_{\beta,\gamma} ~ \subseteq ~ \ran{\iota_\xi} ~ \cup ~ \calA^{c^*}_{\beta,\gamma} ~ \subseteq ~ m^*(a\cup\ran{\iota_\xi}) ~ = ~ m(a)$$ holds for all $\beta,\gamma\in m(a)$ with $\beta\neq\gamma$.     
    %
    %
    %
  \end{proof}

  \begin{claim*}
   If $\zeta<\omega_2$ and $\beta<\alpha<\alpha_\zeta$, then $\calA^c_{\alpha,\beta}\subseteq\alpha_\zeta$. 
  \end{claim*}
  
  \begin{proof}[Proof of the Claim]
   Fix $\alpha_\zeta\leq\gamma<\omega_2$. Let $\xi<\omega_2$ be minimal with $\gamma<\alpha_\xi$ and let $d$ be the unique finite subset of $\omega_2+\omega$ such that the domain of $e_\xi$ is equal to $[d]^2$. 
   
   First, assume that $\xi$ has Type $0$. Then $\alpha_\xi=\gamma+1$ and the above constructions ensure that $$s_\gamma(p_1(c(\alpha,\gamma))) ~ = ~ \alpha ~ \neq ~ \beta ~ = ~ s_\gamma(p_1(c(\beta,\gamma)))$$ holds. 
   %
   %
   This allows us to conclude that $c(\alpha,\gamma)\neq c(\beta,\gamma)$ holds in this case. 
   
   Next, assume that $\xi$ has Type $1$. Set $\alpha_*=\sup_{\eta<\xi}\alpha_\eta$ and let $\map{c^*}{[\alpha_\xi]^2}{\omega_1}$ denote the function used in the construction of $c_\xi$. 
  Then $\alpha<\alpha_\zeta\leq\alpha_*\leq\gamma<\alpha_\xi$ and hence $\gamma\in\ran{\iota_\xi}$. Moreover, the above computations show that $c^*(\alpha,\gamma)\neq c^*(\beta,\gamma)$. 
   Set $F=\ran{p_0\circ e_\xi}$ and $n=\betrag{d\setminus\omega_2}$. 
   Now, if $\alpha,\beta\in\ran{\iota_\xi}$, then $\alpha,\beta\in d\cap\alpha_*$ and, since $\calA^{e_\xi}_{\alpha,\beta}\subseteq\omega_2$,  our construction ensures that $c(\alpha,\gamma)\neq c(\beta,\gamma)$. 
  Next, if $\alpha\in\ran{\iota_\xi}$ and $\beta\notin\ran{\iota_\xi}$, then we have $c(\alpha,\gamma)\in\ran{e_\xi}$, $p_0(c(\beta,\gamma))=p_0(c^*(\beta,\gamma))\notin\ran{p_0\circ e_\xi}$ and therefore we know that $c(\alpha,\gamma)\neq c(\beta,\gamma)$.    
  The same argument shows that, if $\alpha\notin\ran{\iota_\xi}$ and $\beta\in\ran{\iota_\xi}$, then $c(\alpha,\gamma)\neq c(\beta,\gamma)$. 
  Finally, if $\alpha,\beta\notin\ran{\iota_\xi}$, then  $c(\alpha,\gamma)=c^*(\alpha,\gamma)\neq c^*(\beta,\gamma)=c(\beta,\gamma)$.  
  \end{proof}

  \begin{claim*}
   If $a\in[\omega_2]^{{<}\omega}$ and  $\alpha,\beta\in m(a)$ with $\alpha\neq\beta$, then $\calA^c_{\alpha,\beta}\subseteq m(a)$. 
  \end{claim*}
  
 \begin{proof}
  Pick $\xi<\omega_2$ with $m(a)\subseteq\alpha_\xi$. Then the previous claim shows that $\calA^c_{\alpha,\beta}\subseteq\alpha_\xi$ and we can use our first claim to conclude that $\calA^c_{\alpha,\beta}=\calA^{\alpha_\xi}_{\alpha,\beta}\subseteq m(a)$. 
 \end{proof}

  \begin{claim*}
   Given   a finite subset $d$ of $\omega_2+\omega$, a function  $\map{e}{[d]^2}{\omega_1}$ and a function $\map{s}{d}{\omega_2}$ such that  $e\restriction[d\cap\omega_2]^2 =  c\restriction[d\cap\omega_2]^2$, $r\restriction(d\cap\omega_2)=s\restriction(d\cap\omega_2)$ and   $\calA^e_{\alpha,\beta} \subseteq \omega_2$ for all $\alpha,\beta\in d\cap\omega_2$ with $\alpha\neq\beta$,  there exists an injection $\map{\iota}{d}{\omega_2}$ with $\iota\restriction(d\cap\omega_2)=\id_{d\cap\omega_2}$, $r(\iota(\alpha))=s(\alpha)$ for all $\alpha\in d$ and $c(\iota(\alpha),\iota(\beta))  =  e(\alpha,\beta)$ for all $\alpha,\beta\in d$ with $\alpha\neq\beta$. 
  \end{claim*}
  
  \begin{proof}[Proof of the Claim]
   Without loss of generality, we may assume that $d\setminus\omega_2\neq\emptyset$. 
   The above choices ensure that we can find $\xi<\omega_2$ with the property that $e=e_\xi$, $s=s_\xi$ and $d\cap\omega_2\subseteq\alpha_*=\sup_{\zeta<\xi}\alpha_\zeta$. 
   We define  $c_*=c\restriction[\alpha_*]^2=\bigcup_{\zeta<\xi}c_\zeta$ and $r_*=r\restriction\alpha_*=\bigcup_{\zeta<\xi}r_\zeta$. 
   We then know that $d\nsubseteq\alpha_*$, $d\cap\omega_2\subseteq\alpha_*$, $c_*\restriction[d\cap\omega_2]^2=e_\xi\restriction[d\cap\omega_2]^2$, $r_*\restriction(d\cap\omega_2)=s_\xi\restriction(d\cap\omega_2)$ and $\calA^{e_\xi}_{\alpha,\beta}\subseteq\omega_2$ for all $\alpha,\beta\in d\cap\omega_2$ with $\alpha\neq\beta$. 
   In particular, this shows that $\xi$ has Type $1$ and $\map{\iota_\xi}{d}{\alpha_\xi}$ is an injection with $\iota_\xi\restriction(d\cap\omega_2)=\id_{d\cap\omega_2}$, $$r(\iota_\xi(\alpha)) ~ = ~ r_\xi(\iota(\alpha)) ~ = ~ s_\xi(\alpha)$$ for all $\alpha\in d$ and $$c(\iota_\xi(\beta),\iota_\xi(\gamma)) ~ = ~ c_\xi(\iota_\xi(\alpha),\iota_\xi(\beta)) ~ = ~ e_\xi(\alpha,\beta)$$ for all $\alpha,\beta\in d$ with $\alpha\neq\beta$. 
  \end{proof}
  
  This completes the proof of the theorem. 
\end{proof}


\section{Concluding remarks and restating the problem}

We summarize the current situation of the problem motivating the results of this paper: Hjorth proved that there exists some countable model $\M$ that belongs to the constructible universe $\LL$ and which characterizes $\aleph_1$ in all transitive models of $\ZFC$. Using the Scott sentence of $\M$, he constructed two complete sentences, call them $\sigma_1$ and $\sigma_2$, using what we called the first and the second Hjorth construction. Moreover, in all transitive models of $\ZFC$, exactly one of these sentences characterizes $\aleph_2$. 
If $\CH$ holds, then $\sigma_1$ characterizes $\aleph_1$ and $\sigma_2$ characterizes $\aleph_2$. If $(\diag)$ holds (and $\CH$  necessarily fails by Proposition \ref{noch}), then $\sigma_1$ characterizes $\aleph_2$ and $\sigma_2$ characterizes $\aleph_3$.

Therefore, Hjorth's solution to the problem of characterizing $\aleph_2$ is dependent on the underlying model of set theory. One may ask whether the same holds true for $\aleph_3$ and, in general, for successor $\aa$ with $2<\alpha<\omega_1$. For $\alpha<\omega$, this is easily seen to be true, because Hjorth's characterization of $\aleph_3$ uses inductively the characterization of $\aleph_2$ etc. For $\alpha>\omega$ our construction does not yield an answer. One would have to extend our results for functions from $\omega_1$ to $\omega_1$ into results for functions from $\omega_{\omega+1}$ to $\omega_{\omega+1}$. However, we think the main question here is how to characterize $\aleph_\alpha$, $\alpha<\omega_1$, \emph{in an absolute way}. To make things precise: 

\begin{question}\label{openquestion}
Does there exist a formula $\Phi(v_0,v_1)$ in the language of set theory such that $\ZFC$ proves the following statements hold for all ordinals $\alpha$: 
 \begin{enumerate}
  \item In $\LL$, there exists a unique code\footnote{Using some canonical G\"odelization of $\calL_{\kappa,\omega}$-formulas.} $c$ for a complete $\calL_{\alpha^+,\omega}$-sentence  $\psi_\alpha$ such that $\Phi(\alpha,c)$ holds. 
  
  \item If $\alpha$ is countable and $\psi_\alpha$ is as above, then $\psi_\alpha$ characterizes $\aleph_\alpha$. 
 \end{enumerate}
\end{question}

As we mentioned this is true for limit ordinals $\alpha$. In \cite{BKLdap}, the authors provide a characterization of all $\aleph_n$, for $n$ finite, that is absolute in the way described above. For successor ordinals $\alpha>\omega$ the question remains open. 



 Another canonical way to formulate the existence of absolute characterizations is given by  \emph{Shoenfield absoluteness} (see {\cite[Theorem 13.15]{MR1994835}}) and the fact that $\Sigma^1_3$-statements are upwards absolute between transitive models of set theory with the same ordinals.

\begin{question}
 Is there a $\Sigma^1_3$-formula $\Phi(v_0,v_1)$ in the language of second-order arithmetic with the property that the axioms of $\ZFC$ prove that the following statements hold: 
 \begin{enumerate}
     \item For every real $a$, there is a unique real $b$ such that $\Phi(a,b)$ holds. 
     
     \item If $\alpha$ is a countable ordinal, $c$ is a code for a complete $\calL_{\omega_1,\omega}$-sentence that characterizes $\aleph_\alpha$ and $d$ is a real with the property that  $\Phi(c,d)$ holds, then $d$ is a code for a complete $\calL_{\omega_1,\omega}$-sentence that characterizes $\aleph_{\alpha+1}$. 
 \end{enumerate}
 
     
\end{question}

 %
 Note that, since it is possible to force $\CH$ over a model of $(\diag)$ without adding new real numbers, the results of this paper show that the property of an $\calL_{\omega_1,\omega}$-sentence to characterize  $\aleph_2$ is not absolute between models of set theory with the same real numbers.

In the light of results of Woodin in \cite{MR2723878} that show the existence of a proper class of Woodin cardinals implies that the theory of  $\LL(\RRR)$ with real parameters  is generically absolute, it also seems natural to consider the following question:

\begin{question}
 Is there a formula $\Phi(v_0,v_1)$ in the language of set theory with the property that the theory $\ZFC+\anf{\textit{There exists a proper class of Woodin cardinals}}$ proves the following statements hold: 
 \begin{enumerate}
     \item For every real $a$, there is a unique real $b$ such that $\Phi(a,b)$ holds in $\LL(\RRR)$. 
     
     \item  If $\alpha$ is a countable ordinal, $c$ is a code for a complete $\calL_{\omega_1,\omega}$-sentence that characterizes $\aleph_\alpha$ and $d$ is a real with the property that  $\Phi(c,d)$ holds in $\LL(\RRR)$, then $d$ is a code for a complete $\calL_{\omega_1,\omega}$-sentence that characterizes $\aleph_{\alpha+1}$.
 \end{enumerate} 
\end{question}

 We end this paper by considering the question whether versions of the combinatorial principle $(\diag)$ can hold at  cardinals larger than $\omega_1$. Note that many of the techniques used in the consistency proofs of Sections \ref{diagonal} and \ref{nobpfa} have no direct analogs at higher cardinals. 
 The following question considers two interesting test cases for such generalizations.

\begin{question}
  Are the following statements consistent with the axioms of $\ZFC$?
  \begin{enumerate}
   \item For every sequence $\seq{\map{f_\alpha}{\omega_2}{\omega_2}}{\alpha<\omega_2}$ of functions, there exists a function $\map{g}{\omega_2}{\omega_2}$ with the property that the set $\Set{\xi<\omega_2}{f_\alpha(\xi)=g(\xi)}$ is finite for every $\alpha<\omega_2$. 
   
   \item For every sequence $\seq{\map{f_\alpha}{\omega_\omega}{\omega_\omega}}{\alpha<\omega_\omega}$ of functions, there exists a function $\map{g}{\omega_\omega}{\omega_\omega}$ with the property that the set $\Set{\xi<\omega_\omega}{f_\alpha(\xi)=g(\xi)}$ is finite for every $\alpha<\omega_\omega$. 
  \end{enumerate}
\end{question}


 \bibliographystyle{plain}
\bibliography{references}

\end{document}